\documentclass[letterpaper]{article}%{amsart}

\usepackage{amsfonts,amssymb,amsthm, amsmath}
\usepackage{titlefoot}

\theoremstyle{plain}
\newtheorem{theorem}{Theorem}%[subsection]
\newtheorem{proposition}[theorem]{Proposition}
\newtheorem{lemma}[theorem]{Lemma}
\newtheorem{corollary}[theorem]{Corollary}
\theoremstyle{definition}
\newtheorem{definition}[theorem]{Definition}
\theoremstyle{remark}
\newtheorem*{remark}{Remark}

\newcommand{\N}{\mathbb{N}}
\newcommand{\R}{\mathbb{R}}
\newcommand{\C}{\mathbb{C}}
\newcommand{\de}{\mathrm{d}}
\newcommand{\dc}{{\mathrm{d}}^c}

\newcommand{\lie}[1]{\text{Lie}({#1})}

\newcommand{\gtube}{{G^\star}}
\newcommand{\modmod}{/\!/}
\newcommand{\timescat}[1]{\times^{\modmod{#1}}}
\newcommand{\lieder}{\mathcal{L}}
\newcommand{\derv}[1]{\frac{\partial}{\partial #1}}
\newcommand{\osheaf}{\mathcal{O}}
\newcommand{\Sym}{\operatorname{Sym}}
\newcommand{\Hess}{\operatorname{Hess}}
\newcommand{\Levi}{\operatorname{L}}
\newcommand{\symtensor}{{\bar\otimes}}

\title{Invariant K\"ahler
potentials and symplectic reduction}
\author{P. Heinzner \& B. Stratmann}
\date{}

\begin{document}
\maketitle
\unmarkedfntext{The research for this work was partially supported by the
SFB/Transregio 191 
``Symplektische Strukturen in Geometrie, Algebra und Dynamik''.} 
\begin{abstract}
For a proper Hamiltonian action of a Lie group $G$ on a K\"ahler manifold $(X,\omega)$ with momentum map $\mu$ we 
show that the symplectic reduction
$\mu^{-1}(0)/G$ is a normal complex space.
Every point in $\mu^{-1}(0)$
has a $G$-stable open neighborhood
on which $\omega$ and $\mu$ are given by a
$G$-invariant K\"ahler potential.
This is used to show that
$\mu^{-1}(0)/G$ is a K\"ahler space.
Furthermore we examine the existence of
potentials away from
$\mu^{-1}(0)$
with both positive and negative results.
\end{abstract}
%\section{Introduction}
Let $(X,\omega)$ be a K\"ahler manifold and $G$ a closed Lie subgroup of the
group $\operatorname{Iso}_\mathcal{O}(X)$ of holomorphic isometries of the K\"ahler structure. Since $G$ is closed
in $\operatorname{Iso}_\mathcal{O}(X)$, the $G$-action is proper. We will assume that the action of
$G$ on $X$ is Hamiltonian, meaning
that a momentum map
$\mu:X\to\lie{G}^*$ is given. This is a
$G$-equivariant map with
$\lie{G}^*$ endowed with the coadjoint action
such that
for each $\xi\in\lie{G}$ the corresonding 
component function
$\mu^\xi$ satisfies
$\de\mu^\xi=\imath_{\tilde\xi}\omega$. 
We will refer to $(X,G,\omega,\mu)$ as a
Hamiltonian K\"ahler manifold.
The symplectic reduction of $X$ is defined as the quotient $\mathcal{M}/G$ where $\mathcal{M}=\mu^{-1}(0)$.
On $\mathcal{M}$ we have the sheaf $\mathcal{O}_\mathcal{M}=\mathcal{O}_X\vert_\mathcal{M}$ where
$\mathcal{O}_X$ denotes the sheaf of holomorphic functions on $X$. Sections of $\mathcal{O}_\mathcal{M}$
are given by functions on $\mathcal{M}$
which are restrictions of holomorphic functions
and we call them
holomorphic functions on $\mathcal{M}$. 
Let $p:\mathcal{M}\to\mathcal{M}/G$ denote the quotient map. We have the direct image sheaf
$p_*\mathcal{O}_\mathcal{M}$
on $\mathcal{M}/G$ and the corresponding sheaf $p_*\mathcal{O}_\mathcal{M}^G=(p_*\mathcal{O}_\mathcal{M})^G$
of $G$-invariant sections,
denoted $\mathcal{O}_{\mathcal{M}/G}$.
In other words we associate to an open subset $Q\subset\mathcal{M}/G$ the algebra
$\mathcal{O}_\mathcal{M}(p^{-1}(Q))^G$ of invariant holomorphic functions
on $\mathcal{M}$, these are $G$-invariant continuous
functions on $p^{-1}(Q)$ that admit for 
each orbit $G\cdot x\subset p^{-1}(Q)$ a neighborhood in $X$
to which $f$ extends as a holomorphic function, in fact it extends as a $G$-invariant
holomorphic function on a $G$-stable neighborhood, as the proof will show.
Our first result is the following.
\begin{theorem}
$(\mathcal{M}/G,\mathcal{O}_{\mathcal{M}/G})
$ is a reduced normal complex space.
\label{theorem_quotientComplexStructure}
\end{theorem} %p_*\mathcal{O}_\mathcal{M}^G)
If $0\in\lie{G}^*$ is
a regular value of $\mu$, then the
theorem follows almost directly
from the construction of the symplectic reduction
$\mathcal{M}/G$. One just has to verify
that the natural almost complex structure
on the quotient is integrable.
In the case where $G$ is a semisimple
Lie group the result has been obtained in the
PhD~thesis of M.~Ammon \cite{Amm97}
and has been extended to Lie groups
which are contained in some complex
Lie group in the PhD~thesis of A.~Kurtdere~\cite{Kur09}.
In Theorem~\ref{theorem_quotientComplexStructure}
we allow $G$ to be any Lie group.
Furthermore we are able to realize
$\mathcal{M}/G$ locally more explicit as
a quotient $U\modmod G$
for some suitable $G$-stable
neighborhood $U$ in $X$. This quotient
$U\modmod G$ is an analytic Hilbert quotient, i.e. two points are identified if they cannot be
separated by invariant holomorphic functions.
The quotient is endowed with the structure
sheaf $\mathcal{O}_{U\modmod G}$
defined by $\mathcal{O}_{U\modmod G}(V)=\mathcal{O}_U^G(\pi^{-1}(V))$
for every open set $V\subset U\modmod G$
(see
Definition~\ref{definition_HilbertQuotient}).\\
The quotient $p:\mathcal{M}\to\mathcal{M}/G$ is universal in the sense that
any $G$-invariant holomorphic map
factorizes over $\mathcal{M}/G$.
\begin{theorem}
For every $x\in\mathcal{M}$ there exists a
$G$-stable open neighborhood $U$
such that the analytic Hilbert quotient
$\pi:U\to U\modmod G$ realizes 
$(U\modmod G,\mathcal{O}_{U\modmod G})$
as a reduced normal complex space
and the inclusion
$i_U:\mathcal{M}_U=\mathcal{M}\cap U\hookrightarrow U$
induces a biholomorphic map
$\varphi_U:\mathcal{M}_U/G\to U\modmod G$
such that the following diagram commutes.
\begin{eqnarray}
\mathcal{M}_U&\stackrel{i_U}{\hookrightarrow}& U\notag\\
\downarrow p&\quad&\downarrow\pi\label{eqn_localRealizationInTheorem}\\
\mathcal{M}_U/G&\stackrel{\varphi_U}{\longrightarrow}& U\modmod G\notag
\end{eqnarray}
\label{theorem_quotientLocallyGivenFromHilbertQuotient} 
\end{theorem}
For a $G$-invariant smooth strictly plurisubharmonic function $\rho$ one can define a Hamiltonian
K\"ahler structure by
$$\omega=\de\dc\rho\qquad\text{and}\qquad\mu^v=-\imath_{\tilde v}\dc\rho\quad\text{for each}\quad v\in\lie{G}$$
where $\tilde v$ is the induced vector field on $X$ and $\mu^v(x)=\mu(x)(v)$ with the real operator
$\dc=\frac{1}{2i}(\partial-\bar\partial)$.
In this case we will say that $\mu$ is given by a $G$-invariant potential $\rho$ (Definition~ \ref{definition_potential}).
In general there is no hope to obtain a global potential
for $\mu$. Our second result says that locally around each point in $\mathcal{M}$
this is the case.
\begin{theorem}
For every point $x\in \mathcal{M}$ there exists an open
$G$-stable neighborhood $Y$ of $x$ in $X$ and a $G$-invariant
potential $\rho:Y\to\R$ for 
which $\omega\vert_Y$ and $\mu\vert_Y$ are given by $\rho$.\label{theorem_localPotentialAroundM}
\end{theorem}
The manifold $X$ has a stratification
by $G$-orbit type. Points $x_1$ and $x_2$ belong to the same stratum
if their $G$-orbits are $G$-equivariantly diffeomorphic, i.e.
if their isotropy groups are conjugate in $G$.\\
The stratification on $X$ induces a stratification on $\mathcal{M}$.
The intersection of $\mathcal{M}$
with a stratum gives a local submanifold. This stratification pushes down to a stratification
of the quotient $\mathcal{M}/G$.
It is not difficult to see that the K\"ahler form $\omega$ pulled back to $\mathcal{M}$
induces a K\"ahler form $\bar\omega$ on each stratum in $\mathcal{M}/G$ (s.~Section~\ref{section_KaehlerStructure} or \cite{SjaLer91} for a different proof).\\
We will show that the smooth K\"ahler forms
on the strata extend uniquely to a K\"ahler structure on the complex space $\mathcal{M}/G$. By definition
a K\"ahler structure on a complex space $Q$
is given by an open covering $Q=\bigcup Q_\alpha$
and strictly plurisubharmonic
continuous functions $\rho_\alpha:Q_\alpha\to \R$
such that on $Q_\alpha\cap Q_\beta$
the difference $\rho_\alpha-\rho_\beta$ is locally the real part of a holomorphic function
(Definition~\ref{definition_KaehlerStructure}).\begin{corollary}
There is a unique K\"ahler structure $\bar\omega$ on $\mathcal{M}/G$ such that
$ p^*\bar\omega=i_\mathcal{M}^*\omega$ with $i_\mathcal{M}:\mathcal{M}\to X$ denoting the inclusion. The continuous strictly plurisubharmonic
functions $\bar\rho_\alpha$ defining $\bar\omega$
are smooth restricted to each stratum.
\label{corollary_KaehlerStructure}
\end{corollary}
More precisely, for every point $y\in\mathcal{M}/G$ there is an open $G$-stable
neighborhood $U$ of $ p^{-1}(y)$ in $X$ and a $G$-invariant K\"ahler
potential $\rho$ of $\omega$ on $U$ such that
the restriction $\rho\vert_{\mathcal{M}\cap U}$
induces a continuous function $\bar\rho: p(\mathcal{M}\cap U)\to\R$ which is
a K\"ahler potential
for $\bar\omega$ on each non-singular stratum of $\mathcal{M}/G$, i.e. $\bar\omega=\de\dc\bar\rho$.\\
Existence of local $G$-invariant potentials
can also be shown in $G$-stable neighborhoods
of totally real orbits under an additional condition. In Section~\ref{sect_localPotential} we show the following theorem.
Denote $G^0$ the connected component of the neutral element
of $G$.
\begin{theorem}
Let $(X,G,\omega,\mu)$ be a Hamiltonian manifold with $H^1_{dR}(G^0)=0$. Let
the orbit $G\cdot x$ be totally real and
$\omega$ restricted to this orbit be exact. Then there exists an invariant local
potential on some $G$-stable neighborhood of $x$.\label{theorem_localPotentialAwayMomentZero}
\end{theorem}
In the last section we give
examples where there is no local invariant potential away from the momentum zero level (Section~\ref{sect_noPotential}).

%%%%%%%%%%%%%%%%%%%%%%%%%%%%%%%%%%%%%%%%%%%%%
%
%   1   Complexifying the group and the orbit
%
%%%%%%%%%%%%%%%%%%%%%%%%%%%%%%%%%%%%%%%%%%%%%

\section{Complexified group and orbits}
Let $G$ be a Lie group and $X$ a complex $G$-manifold.
By this we mean that $G$ acts smoothly by holomorphic
transformations.\\
In many cases of interest the $G$-action on $X$ is given by restricting a holomorphic action
of the complexified group $G^\C$ on a complex
manifold $Z$ which contains $X$ as an open $G$-stable
subset. Provided $G$ is contained in $G^\C$
this is not always the case. If 
otherwise $G$ is not contained in $G^\C$
this is even never the case for an effective $G$-action, e.g.
consider the group $H=SL_2(\R)$ acting by left multiplication on $H^\C=SL_2(\C)$. This action is proper and free and therefore there is an open
$H$-stable neighborhood $X$ of $H$ in $H^\C$
which is $H$-equivariantly diffeomorphic to $H\times \Sigma$ where $\Sigma\subset \R^3$ denotes a ball
in $\R^3$. The
universal covering $G=\hat H$ of $H$
acts holomorphically on the universal
covering $\hat X$
of $X$ and this $G$-action
cannot be realized as a restriction of a holomorphic action of $G^\C=H^\C=SL_2(\C)$.
\subsection{The $G$-tube}\label{subsection_Gtube}
As a first step we complexify the group $G$.
As these complexications are desired to be shrinkable
we need them to fulfill the Runge property relative
to each other.
\begin{definition}
An open subset $Y$ of a complex manifold
$X$ is said to be {\em Runge}
if the restriction map
$\osheaf(X)\to\osheaf(Y)$ has dense image in
$\osheaf(Y)$. 
\label{definition_runge} 
\end{definition}
\begin{definition}
Let $G$ act on $X$.
A continuous function $\rho:X\to\R$
is called {\em $G$-ex\-haus\-tive} or a
{\em $G$-exhaustion} if it is $G$-invariant 
and the sets
\mbox{$\{x\in X\ \vert\ \rho(x)<c\}/G$}
are relatively compact in $X/G$
for all $c<\sup\rho$. We just write {\em ex\-haus\-tion/ex\-haus\-tive} if the trivial group acts.
\label{definition_exhaustion}
\end{definition}
Recall that classical results from Grauert
show that given a strictly plurisubharmonic
exhaustion $\rho:X\to\R^{\geq 0}$
the relatively compact sets
$X_c=\{x\in X\ \vert\ \rho(x)<c\}$ are Runge
(cf. e.g. \cite[Thm. 1.3 (v)]{HenLei98}).
The existence
of a strictly plurisubharmonic exhaustion on $X$
is equivalent for $X$ to be a Stein manifold.
If we drop the assumption of $\rho$ being an exhaustion
but keep that $X$ is a Stein manifold the Runge property is still satisfied as the following shows.
\begin{lemma}
Let $X$ be a Stein manifold and $\rho_Y:X\to\R^{\geq 0}$ a strictly plurisubharmonic
function. Then $Y=\{y\in Y\ \vert\ \rho_Y(y)<1\}$ is
Runge in $X$.\label{lemma_Runge}
\end{lemma}
\begin{proof}
We fix some strictly plurisubharmonic exhaustion function $\rho:X\to\R^{\geq 0}$ and for $n\in\N$
we set
$\rho_n=\max\{\frac{1}{n}\rho,\rho_Y\}$. The function
$\rho_n$
a strictly plurisubharmonic in the sense of perturbation and an exhaustion. It follows that
$$Y_n=\{x\in X\ \vert\ \rho_n(x)<1\}$$
is an open Stein submanifold of $X$. Moreover
$Y_n$ is Runge in $X$.
For $n>1$ we set
$$K_n=\{x\in X\ \vert\ \rho(x)\leq n-\frac{1}{2}, \rho_Y(x)\leq 1-\frac{1}{n}\}\ .$$
These are compact subsets of $Y_n$
with $K_n\subset K_{n+1}$ and
$\bigcup_{n\geq 2} K_n=Y$.
For every
holomorphic function $f\in \mathcal{O}(Y)$
there is a function $f_n\in\mathcal{O}(X)$ such that \mbox{$\Vert f-f_n\Vert_{K_n}<\frac{1}{n}$} holds
in the sup-norm on $K_n$. This shows that
$f_n\vert_Y$ converges to $f$, i.e. $\mathcal{O}(X)$
has 
dense image in $\mathcal{O}(Y)$.
\end{proof}
\begin{definition}
A Stein $G$-manifold $\gtube$ is said to be a {\em $G$-tube} if
\begin{enumerate}
\item there is a $G$-equivariant embedding $\kappa:G\to\gtube$
where $G$ acts on $G$ by left multiplication,
\item every $G$-orbit in $\gtube$ is a totally real submanifold of $\gtube$,\label{condition_tubeSplitsIntoTotallyRealGorbits}
\item there is a real
submanifold $\Sigma\subset \gtube$
with $\kappa(e)\in\Sigma$ such that
\begin{enumerate}
\item the map $G\times \Sigma\to\gtube,
(g,\gamma)\mapsto g\cdot \gamma$ is a $G$-equivariant diffeomorphism,
\item $\Sigma$ is diffeomorphic to a convex bounded
open neighborhood of $0$ in $\R^k$ with $k=\dim G$,
\end{enumerate}
\item there is a strictly plurisubharmonic 
$G$-exhaustion
$\rho:\gtube\to \R$ which
becomes minimal on $\kappa(G)$.
\item the manifold $\gtube$ is ``shrinkable'', i.e.
every $G$-stable neighborhood $U$ of $\kappa(G)$
contains a Stein $G$-submanifold $T\supset\kappa(G)$
such that $\Sigma_T=\Sigma\cap T$ is convex,
$\rho\vert_{\Sigma_T}$ is an exhaustion
and $T$ is Runge in $\gtube$.
\end{enumerate}
\end{definition}
Condition~\eqref{condition_tubeSplitsIntoTotallyRealGorbits}
implies that the $G$-action on $\gtube$ is proper and free.
We will refer to $\Sigma$ as a ``slice''
for the $G$-action on $\gtube$
and will identify $\kappa(G)$ with $G$.\\
A smooth function $\rho$
on a complex manifold $X$ is strictly plurisubharmonic if and only if
$\omega=i\partial\bar\partial\rho=\de\dc\rho$
is a K\"ahler form. Note that
$\dc=\frac{1}{2i}(\partial-\bar\partial)$ is a real
operator.\\
Later we will use $G$-tubes to construct 
local slice models
of the $G$-action on a complex manifold $X$
around totally real $G$-orbits.
For this we have to fix a compact subgroup
$K$ of $G$ which will be given as an
isotropy group of the $G$-action later.
\begin{theorem}
\textbf{\em (Existence of a  $G$-tube)} %$K$-stable
For every Lie group $G$ with fixed compact subgroup
$K$ there exists a $G$-tube $\gtube$
such that the $(G\times K)$-action
on $G$ given by $(g,k)\cdot x=gxk^{-1}$
extends to $\gtube$ as an action by holomorphic transformations.\label{theorem_gtubeExistence}
\end{theorem}
\begin{proof}
We view $G$ as a $(G\times K)$-manifold. Then
there is 
a complex $(G\times K)$-manifold $T$ which
is $(G\times K)$-equivariantly
diffeomorphic
to a $(G\times K)$-stable neighborhood
of the zero section of the tangent bundle $TG$ of $G$.
In order to see this
one has to choose a real analytic $(G\times K)$-invariant
Riemannian metric
on $G$ and apply the main results in \cite{LemSko91,GuiSte92}.
By a result of Winkelmann (\cite{Win93}) the manifold
$T$ which we identify with a subset of $TG$
can be chosen to be a Stein manifold (see \cite{HHK96}).
On $TG$ we have the function $\rho(v)=\Vert v\Vert^2$,
$v\in T_gG$,
where $\Vert\ .\ \Vert$ is defined by the 
$(G\times K)$-invariant metric on $G$.
The function is strictly convex on $T_gG$ and constant
on the zero section, a totally real
submanifold of maximal dimension.
Thus $\rho$ is strictly plurisubharmonic on a neighborhood of the zero section in $T$. As $\rho$ is $(G\times K)$-invariant this
neighborhood can be chosen $(G\times K)$-stable.
This shows that for some $c_0>0$ and
all $0 < c \leq c_0$
the $(G\times K)$-stable sets
$T_c=\{v\in T\ \vert\ \rho(v)<c\}$
intersect $T_eG$ in a convex bounded set
and with $\Sigma=T_{c_0}\cap T_eG$
the function
$\rho\vert_\Sigma$ is strictly convex.
The restriction
$\rho\vert_{T_c}$
is a $G$-exhaustion by construction
becoming minimal on $\kappa(G)$.
Since $T$ is a Stein manifold and the boundary
of $T_c\subset T$ is pseudoconvex,
the $(G\times K)$-stable neighborhoods of
the zero section $G\times \{0\}\subset T$ are open
Stein submanifolds of $T$ (\cite{DocGra60}).
Finally, for every $G$-stable neighborhood $U$ of
$\kappa(G)$ there is a $0<c_1<c_0$
such that $T_{c_1}\subset U$. Lemma~\ref{lemma_Runge}
shows that $T_{c_1}$ is Runge in $T_{c_0}$.
Thus $T_{c_1}$ fulfills all the desired properties.
\end{proof}
The next proposition shows that we can consider a $G$-tube the prototype
set to which the real analytic orbit map extends as a holomorphic map.
\begin{proposition}
Given a complex manifold $X$ and a Lie group $G$ acting by holomorphic transformations, let
$\varphi:G\to X, g\mapsto g\cdot x_0$ be the orbit map for some point $x_0\in X$. Then
$\varphi$ extends to some $G$-tube $\gtube$ as a holomorphic map $\varphi^c:\gtube\to X$.
\label{proposition_orbitMapExtendsToGtube}
\end{proposition}
\begin{proof}
The map $\varphi$ being real analytic extends as a holomorphic map $\psi$ to some neighborhood
$\Omega\subset\gtube$ 
of $e\in G$. We may assume that $\Omega_\Sigma=\Omega\cap\Sigma$ is connected.
Define $T=G\cdot \Omega_\Sigma$ which is $G$-equivariantly diffeomorphic
to $G\times \Omega_\Sigma$.
So, we can define the $G$-equivariant map $\varphi^c:T\to X,
g\cdot x\mapsto g\cdot\psi(x)$. For each point 
$x\in\Omega_\Sigma$ there is a neighborhood $U(x)$ such that
the map $\psi$ is locally $G$-equivariant on $U(x)$. So the maps $\varphi^c$ and $\psi$
are locally $G$-equivariant and identical
on $\Omega_\Sigma$, thus they are locally identical. Hence
$\varphi^c$ is holomorphic in a neighborhood of $\Omega_\Sigma$ and globally by $G$-equivariance.
Restricted to $G\subset \gtube$ both maps $\varphi$ and $\varphi^c$ coincide by construction.
Hence $\varphi^c$ extends $\varphi$ as a holomorphic function on a $G$-stable neighborhood of $G$.
By definition the neighborhood $T$ can be shrunk to a $G$-tube.
\end{proof}
\subsection{Complexified orbits}\label{subsect_cplxfdOrbit}
We will now establish the notion
of a complexified $G$-orbit
for an action of a real 
group by holomorphic transformations. It is meant to be a substitute for
a $G^\C$-orbit in the case where $G^\C$ does not act.
\begin{definition}
For each point $x\in X$, the orbit map
$\varphi: G\to X, g\mapsto g\cdot x$ is real analytic and extends
to a holomorphic map $\varphi^c:\gtube\to X$ on some
$G$-tube.
We denote by $\tau_G$ the smallest topology containing
the sets $\varphi^c(\gtube)\subset X$ as open sets for all orbit maps $\varphi$ and
all $G$-tubes to which $\varphi$ extends.
\end{definition}
Observe that Proposition~\ref{proposition_orbitMapExtendsToGtube}
guarantees the existence
of the extensions $\varphi^c$
to some $G$-tube.
\begin{lemma}
The images
$\varphi^c(\gtube)\subset X$ used in the definition of the topology form a basis of the
topology $\tau_G$.\label{lemma_basisTopology}
\end{lemma}
The proof of this lemma is given after Lemma~\ref{lemma_openSetsInTube}.
\begin{definition}
The {\em complexified $G$-orbit} of $x\in X$ is the connected component
of $x$ in the topology $\tau_G$ denoted by $B_X(x)$
or just by $B(x)$.\label{definition_cplxOrbit}
\end{definition}
In the following we collect some properties of the 
topology $\tau_G$ and provide the postponed proof of Lemma~\ref{lemma_basisTopology}.\\
In a first step we endow the $G$-tube $\gtube$ itself with the topology $\tau_G$ 
and use this later as a model.
\begin{lemma}
Given a $G$-stable open subset $U\subset \gtube$ and $\gamma \in U$
there is a $G$-tube $T\subset \gtube$
and a holomorphic extension
$\varphi^c:T\to\gtube$
of the orbit map 
$\varphi:G\to \gtube, g\mapsto g\cdot\gamma$ 
such that $\varphi^c(T)\subset U$.
\label{lemma_gstableOpenSetsInTubeAreTauGopen}
\end{lemma}
\begin{proof}
The $G$-orbit $G\cdot\gamma$ is totally real by
definition of the $G$-tube. Proposition~\ref{proposition_orbitMapExtendsToGtube}
shows that $\varphi$ extends holomorphically to some
$G$-tube $T_0\subset\gtube$ as $\varphi^c:T_0\to \gtube$. Since $\varphi$ is injective and immersive,
the extension $\varphi^c$ is an injective immersion on some $G$-stable neighborhood of $G$. This set can be
shrunk to a $G$-tube $T$ by definition.
Since $\varphi^c$ is open for
dimension reasons, $\varphi^c$ maps $T$ biholomorphically to its image in $U$.
\end{proof}
Recall that the $G$-tube $\gtube$ is assumed to be $G$-equivariantly diffeomorphic
to $G\times \Sigma$. We consider $\Sigma$ as a subset of $\gtube$ via its identification with $\{e\}\times \Sigma$.
\begin{lemma}
A subset $U\subset\gtube$ is $\tau_G$-open if and only if it is open and
$G$-stable. Such a set is $\tau_G$-connected if and only if its intersection with $\Sigma$ is connected,
in particular $(\gtube,\tau_G)$ is locally connected and the $G$-tube itself is $\tau_G$-connected.
\label{lemma_openSetsInTube}
\end{lemma}
\begin{proof}
Lemma~\ref{lemma_gstableOpenSetsInTubeAreTauGopen} shows that any open, $G$-stable
subset of $\gtube$ is $\tau_G$-open and vice versa, that each point in a 
$\tau_G$-open set in $\gtube$ contains a $G$-stable, open subset, i.e.
each $\tau_G$-open set in $\gtube$ is $G$-stable and open.\\
Let $U\subset\gtube$ be a $\tau_G$-open set
and cover it by two disjoint
$\tau_G$-open sets. This induces a cover of $U\cap\Sigma$. Since $U$ is $G$-stable, $U$ is $\tau_G$-connected
if and only if $U\cap \Sigma$ is connected. The result follows from the fact, that $\Sigma$ is locally
connected and connected.
\end{proof}
\begin{proof}[Proof of Lemma~\ref{lemma_basisTopology}]
We want to show that the sets $\varphi^c(T)$ form a basis of the topology.
Given two points $x_1, x_2\in X$ and for $i=1,2$ some holomorphic extensions
$\varphi^c_i:T_i\to X$ of the corresponding
orbit maps $\varphi_i:G\to X, g\mapsto g\cdot x_i$ to $G$-tubes $T_1, T_2$
respectively.
Suppose there is a point $y\in \varphi^c_1(T_1)\cap \varphi^c_2(T_2)$.
Lemma~\ref{lemma_gstableOpenSetsInTubeAreTauGopen} provides
extensions $\tilde\varphi^c_i:\tilde T_i\to X$ of the orbit map of $y$ 
with $\tilde\varphi^c_i(\tilde T_i)\subset \varphi^c_i(U_i)$ for $i=1,2$. Since both 
holomorphic maps $\tilde\varphi^c_i$
extend the same map, they coincide on some possibly smaller
$G$-tube $T_0\subset\gtube$ which can
even be chosen to be contained
in $\tilde T_1\cap\tilde T_2$. The latter implies that $\tilde\varphi^c_1(T_0)$
is contained in $\varphi^c_1(T_1)\cap \varphi^c_2(T_2)$.
\end{proof}
The following lemma shows that the equivariant holomorphic maps behave in a
natural way compatible with the $\tau_G$-topology.
\begin{lemma}
Let $G$ act on $X$ and $Y$ by holomorphic transformations and let 
$\psi:X\to Y$ be a $G$-equivariant holomorphic map. Then
$\psi:(X,\tau_G)\to (Y,\tau_G)$ is continuous and open.\label{lemma_topologyCompatibleEquivariantHolomorphism}
\end{lemma}
\begin{proof}
Given a $\tau_G$-open set $U_X\subset X$, set $V_Y=\psi(U_X)$. For a given point $y\in V_Y$
choose $x\in U_X$ in the $\psi$-fiber of $y$. By definition of $\tau_G$, there is
a $G$-stable, open set $U\subset \gtube$ such that
$\varphi^c:U\to U_X$ is a holomorphic extension of the orbit map $g\mapsto g\cdot x$.
Due to $G$-equivariance of $\psi$, the map
$\psi\circ\varphi^c$ is in fact a holomorphic extension of the orbit map
$g\mapsto g\cdot y$ with
$\psi\circ\varphi^c(U)\subset V_Y$. Thus $V_Y$ is $\tau_G$-open and therefore $\psi$ is $\tau_G$-$\tau_G$-open.\\
For a given $\tau_G$-open set $V_Y\subset Y$ set $U_X=\psi^{-1}(V_Y)$ and fix $x\in U_X$.
We choose a holomorphic extension $\varphi^c:U\to X$ of the orbit map $g\mapsto g\cdot x$.
If $(\psi\circ\varphi^c)^{-1}(V_Y)$ admits an open neighborhood of $G\subset\gtube$, we may restrict $\varphi^c$ to
this set and are done, since $U_X$ contains a $\tau_G$-open neighborhood of $x$. So suppose the contrary.
There are sequences $t_n\in\gtube$ and $g_n\in G$ such that $t_n'=g_n\cdot t_n$ 
converges to $e$ and $\psi\circ\varphi^c(t_n)\not\in V_Y$.
The set $V_Y$ being $G$-stable and $\varphi^c$ and $\psi$ being $G$-equivariant, 
we conclude $\psi\circ\varphi^c(t_n')\not\in V_Y$.
But $\psi\circ\varphi^c$
is a holomorphic extension of the orbit map $g\mapsto g\cdot y$ and hence extends
to some neighborhood of $G\subset U$ with image in $V_Y$, since $V_Y$ is assumed $\tau_G$-open.
So, $\psi$ is $\tau_G$-$\tau_G$-continuous.
\end{proof}
\begin{lemma}
The space $(X,\tau_G)$ is locally connected. In particular, the
connected components, namely the complexified $G$-orbits, are $\tau_G$-closed and
$\tau_G$-open.\label{lemma_topologyLocallyConnected}
\end{lemma}
\begin{proof}
For topological spaces in general connected components are closed and for
locally connected topological spaces they are also open.
So, we are left to show that $(X,\tau_G)$ is locally connected.
Given a point $x\in X$ and a holomorphic extension $\varphi^c:T\to X$ of the orbit map of $x$
to a $G$-tube $T$. Then $T$ is $\tau_G$-connected as shown in Lemma~\ref{lemma_openSetsInTube}.
Since $\varphi^c$ is $\tau_G$-$\tau_G$-continuous and $\tau_G$-$\tau_G$-open
(Lemma~\ref{lemma_topologyCompatibleEquivariantHolomorphism}),
the image $\varphi^c(T)$
is $\tau_G$-connected and $\tau_G$-open as well,
hence a $\tau_G$-neighborhood of $x$. Therefore $(X,\tau_G)$ is locally connected.
\end{proof}
\begin{lemma}
Assume that the action of $G$ is proper. Then the quotient topology of $\tau_G$ on $X/G$ is Hausdorff.
\end{lemma}
\begin{proof}
The quotient topology of the classical topology is finer than that of $\tau_G$. But for a proper action
the quotient $X/G$ is Hausdorff
for the first one already.
\end{proof}
We will now show that complexified $G$-orbits coincide with $G^\C$-orbits
in the case where the $G$-action is given as a restriction of a
holomorphic $G^\C$-action.
\begin{lemma}
In the case where $G^\C$ acts holomorphically on $X$ the complexified $G$-orbits are
exactly the $G^\C$-orbits.\label{lemma_GCorbitsAreComplexifiedGorbits}
\end{lemma}
\begin{proof}
Let $B$ be the $G^\C$-orbit of $x\in X$. 
We have to show that $B$ is $\tau_G$-open, $\tau_G$-closed and $\tau_G$-connected.
The holomorphic action map $\alpha:G^\C\times X\to X$
induces a holomorphic orbit map $\alpha_x:G^\C\to X, g\mapsto g\cdot x$
and the $G$-equivariant orbit map
$\varphi:G\to B, g\mapsto g\cdot x$
factorizes
via $i:G\to G^\C$ over $\alpha_x$
i.e. $\varphi=\alpha_x\circ i$.
The map $i$ extends as $G$-equivariant a holomorphic map $i^c$ to some $G$-tube $\gtube$
as $i^c:\gtube\to G^\C$. So there is an extension $\varphi^c=\alpha_x\circ i^c:\gtube\to B$.
The set $\varphi^c(\gtube)$ is by definition a $\tau_G$-open neighborhood of
$x\in B$. Thus $B$ is $\tau_G$-open. 
This implies that every $G^\C$-orbit in $X$
is $\tau_G$-open and since the complement
of a $G^\C$-orbit is a union of $G^\C$-orbits
every $G^\C$-orbit is $\tau_G$-open and $\tau_G$-closed.\\
We are left to show that $B$ is $\tau_G$-connected. Endow $G^\C$ with the topology $\tau_G$.
By the same argument as in Lemma~\ref{lemma_openSetsInTube}
it can be seen that the $\tau_G$-open subsets of $G^\C$ are the $G$-stable open subsets
and that a $\tau_G$-open subset $U\subset G^\C$ is
$\tau_G$-connected if and only if
$U/G\subset G^\C/G$ is connected. Therefore $G^\C$ is $\tau_G$-connected. The map
$\varphi^c:G^\C\to B$
being $\tau_G$-$\tau_G$-continuous (see Lemma~\ref{lemma_topologyCompatibleEquivariantHolomorphism})
the image $B=\varphi^c(G^\C)$ is $\tau_G$-connected, as well.
\end{proof}

%%%%%%%%%%%%%%%%%%%%%%%%%%%%%%%%%%%%%%%%%%%%%
%
%   2   The compact case
%
%%%%%%%%%%%%%%%%%%%%%%%%%%%%%%%%%%%%%%%%%%%%%

\section{The compact case}\label{subsect_CptCase}
In this section we will reformulate the results of \cite{Hei91} in the terminology of $K$-tubes and complexified $K$-orbits
of Sections~\ref{subsection_Gtube} and \ref{subsect_cplxfdOrbit}
in the case of a compact group $K$ acting, which is well-understood. This shall illustrate the geometry, furthermore we
will make decisive use of these known results in the sequel.
\subsection{Key results on quotients of Stein manifolds}
The results of this section are shown in \cite{Hei91}.
\begin{definition}
For a $G$-manifold $Y$ we define the
equivalence relation $y_1\sim y_2$
if $f(y_1)=f(y_2)$ for all $f\in\osheaf^G(Y)$, the algebra of $G$-invariant holomorphic functions.
We call the quotient {\em formal Hilbert quotient}
denoted $\pi:Y\to Y\modmod G$.
We define the sheaf $\osheaf_{Y\modmod G}=\pi_*\osheaf^G_Y$ by associating to
every open subset $V\subset Y\modmod G$ the algebra
$\osheaf^G_Y(\pi^{-1}(V))$. We call
$(Y//G,\osheaf_{Y\modmod G})$ the
{\em analytic Hilbert quotient} if it is a
(reduced) complex space.
\label{definition_HilbertQuotient}
\end{definition}
\begin{theorem}
Let the compact Lie group $K$ act on the Stein manifold $X$
by holomorphic transformations. Then
\begin{enumerate}
\item the quotient $X\modmod K$ endowed with the sheaf $\mathcal{O}_{X\modmod K}$ of $K$-invariant
holomorphic functions on $X$ is
an analytic Hilbert quotient,
\item $(X\modmod K, \mathcal{O}_{X\modmod K})$
is a Stein space, and 
\item any $K$-invariant holomorphic map
$\psi:X\to Y$ from $X$ into a complex space $Y$
factorizes over $X\modmod K$.
\end{enumerate}\label{theorem_quotientStein}
\end{theorem}
\begin{theorem}
Let the compact Lie group $K$ act on a Stein manifold $X$ by holomorphic transformations. Then $X$
can be realized as an open $K$-invariant subset
of a Stein manifold $X^\C$ such that
 \begin{enumerate}
  \item the complexified Lie group $K^\C$ acts holomorphically on $X^\C$
  \item $X^\C=K^\C\cdot X$
  \item $X$ is Runge in $X^\C$.
 \end{enumerate}\label{theorem_complexificationOfAction}
\end{theorem}
\begin{remark}
This result fails in general for a non-compact group $G$ acting.
\end{remark}
\subsection{Complexified orbits and orbits of the complexified\\ group}
In a Stein manifold with a holomorphic $K^\C$-action the complexified $K$-orbits
are exactly the $K^\C$-orbits (Lemma~\ref{lemma_GCorbitsAreComplexifiedGorbits}).
In the context in which a complexification of the space exists as in Theorem~\ref{theorem_complexificationOfAction}
there is another support that the complexified $K$-orbit is the suitable generalization of $K^\C$-orbits.
\begin{proposition}
Let the compact Lie group $K$ act on a Stein manifold $X$ by holomorphic transformations and $X^\C$ be its complexification
in the sense of Theorem~\ref{theorem_complexificationOfAction}. We regard $X$ as a $K$-stable subset of $X^\C$.
Then for any point $x\in X$
\begin{equation*}
K^\C\cdot x\cap X=B_X(x)\label{equation_KCIntersection}
\end{equation*}
and in particular
\begin{equation*}
\overline{K^\C\cdot x}\cap X=\overline{B_X(x)}\label{equation_closureKCorbitsEqualClosureComplexifiedOrbits}
\end{equation*}
The closures are meant in the ambient spaces respectively, i.e. $K^\C\cdot x \subset X^\C$ and $B_X(x)\subset X$.
\label{proposition_complexifiedOrbitIsIntersectionWithKCorbitInComplexifiedSpace}
\end{proposition}
\begin{proof}
Observe first that $X$ is a $K$-stable subset of $X^\C$ such that the topology $\tau_K$ on $X$ is the relative
topology of the topology $\tau_K$ on $X^\C$.
Denote $L=K^\C\cdot x\cap X$. In $X^\C$, each point in the $K^\C$-orbit
$K^\C\cdot x$ admits a $\tau_K$-open neighborhood in the orbit as well as
each point $y\not\in K^\C\cdot x$ admits a $\tau_K$-open neighborhood in the orbit $K^\C\cdot y$ hence not intersecting
with $K^\C\cdot x$. For points $x,y\in X$ the same holds via intersecting the obtained neighborhoods with $X$. Therefore
$L$ is the union of $\tau_K$-connected components. In order to see that $L$ is in fact one single
$\tau_K$-connected component, we need a result from \cite{Hei91}.
The subset $X\subset X^\C$ is shown to be ``orbit-convex'' \cite[Theorem, sect. 3.4]{Hei91}. This implies that
$L/K$ is connected, hence a single complexified $K$-orbit.
This shows $K^\C\cdot x\cap X=B_X(x)$.\\
Consequently $B_X(x)=
K^\C\cdot x\cap X\subset
\overline{K^\C\cdot x}\cap X$
which implies 
$\overline{B_X(x)}\subset
\overline{K^\C\cdot x}\cap X$.
Finally it is shown
in \cite{Hei91} that $Y=\overline{K^\C\cdot x}\cap B_X(x)$ is
a $K$-stable analytic set such that $\dim K^\C\cdot y\cap X < \dim B_X(x)$
for all $y\in Y$. From this 
$\overline{K^\C\cdot x}\cap X\subset \overline{B_X(x)}$ follows.
\end{proof}
\subsection{The fibers of the analytic Hilbert quotient}
In this section we give a geometrical
description of the fibers of the analytic Hilbert
quotient in terms of complexified $K$-orbits.
\begin{theorem}
For every
$p\in X\modmod K$ and every $x\in\pi^{-1}(p)$
\begin{enumerate}
\item $\overline{B_X(x)}$ is analytic and the union of complexified $K$-orbits $B_X(y)$ satisfying
$\dim B_X(y)< \dim B_X(x)$ for each $y\not\in B_X(x)$.\label{item_dimensionInClosureStrictlySmaller}
\item there is exactly one complexified $K$-orbit of lowest dimension, denoted by $E_X(p)$. This complexified
$K$-orbit is closed.\label{item_oneSingleComplexifiedOrbit_lowestDimension_closed}
\item $E_X(p)$ lies in the closure of each complexified $K$-orbit in the fiber, formally
$E_X(p)\subset\overline{B_X(x)}$.
\item $\pi^{-1}(p)$ consists of the points whose complexified $K$-orbit
close up in $E_X(p)$.
\end{enumerate} \label{theorem_fiberStructureAHQcptCase}
\end{theorem}
For a detailed proof we refer to \cite{Hei91}.
If $K^\C$ acts on $X$, the set
$\overline{K^\C\cdot x}$ consists of $K^\C$-orbits with
\begin{equation*}
\dim K^\C\cdot y<\dim K^\C\cdot x\quad\text{for all}\ y\in \overline{K^\C\cdot x}\backslash K^\C\cdot x
\end{equation*}
The complexified $K^\C$-orbits of lowest dimension are closed, otherwise their closure would contain
a complexified $K^\C$-orbit of lower dimension.
The closure of a $K^\C$-orbit contains only one single $K^\C$-orbit of lowest
dimension.
The dimension observation also provides that this closed $K^\C$-orbit of lowest dimension
lies in the closure of each $K^\C$-orbit in $\overline{K^\C\cdot x}$. Thus the theorem holds for $X^\C$.\\
When we combine these observations with
Theorem~\ref{theorem_complexificationOfAction}
and in particular 
Proposition~\ref{proposition_complexifiedOrbitIsIntersectionWithKCorbitInComplexifiedSpace}
we obtain the statements of the above Theorem~\ref{theorem_fiberStructureAHQcptCase}.
\subsection{The free action case}
Later we will show that
every totally real $G$-orbit admits a neighborhood
which
is $G$-equivariantly biholomorphic to
a ``local slice model'' (Section~\ref{sect_localSliceModel}).
These models arise as analytic Hilbert quotients $\gtube\timescat{K}S=(\gtube\times S)\modmod K$
of $\gtube\times S$ with respect to a free action of a compact group $K$.
We collect the main results of free actions of a compact group.
\begin{proposition}
Let $K$ act freely on a Stein manifold $X$ and let $Y$ be an open $K$-stable Runge Stein subset of $X$.
\begin{enumerate}
 \item $X\modmod K$ is non-singular.\label{inProp_nonSingular}
 \item Each fiber of $\pi:X\to X\modmod K$ consists of one single
 complexified $K$-orbit.\label{inProp_singleComplKorbit}
 \item The projection maps $\pi:X\to X\modmod K$ and
 $\pi_Y:Y\to Y\modmod K$ are submersions. \label{inProp_projectionOpen}
 \item $Y\modmod K$ is a complex space and the inclusion $\imath:Y\to X$ induces an open
 embedding $\imath_{\modmod K}:Y\modmod K\to X\modmod K$.\label{inProp_openEmbedding}
 \item $Y\modmod K$ is Runge in $X\modmod K$.
 \item For every $y\in Y$ we have $B_Y(y)=Y\cap B_X(y)$, in particular
 $\pi_Y^{-1}(p)=Y\cap \pi^{-1}(p)$ holds for all $p\in Y\modmod K\subset X\modmod K$.
\end{enumerate}\label{prop_shrinkingUnderFreeAction}
\end{proposition}
The Runge property will be helpful later to
pass from a model to a smaller submodel.
\begin{proof}
Theorem~\ref{theorem_complexificationOfAction} implies that $Y$ and $X$ are open subsets of $X^\C$. The Linearization Lemma
in \cite[sect. 5.1]{Hei91} provides for a free action that for each $p\in X\modmod K\cong X^\C\modmod K^\C$ there
is a point $x\in X$ with $\pi(x)=p$
and a local complex manifold $S\subset X$ containing $x$ such that $K^\C\times S\to X, (k,s)\mapsto k\cdot s$
is an open embedding. Hence $X^\C$ is a
$K^\C$-principal bundle over the smooth quotient $X\modmod K\cong X^\C\modmod K^\C$.
Each fiber of $\pi:X\to X\modmod K$ is in fact the intersection of $X$ with the fiber of $\pi^\C:X^\C\to X^\C\modmod K^\C$ which
is a single closed $K^\C$-orbit. Hence
Proposition~\ref{proposition_complexifiedOrbitIsIntersectionWithKCorbitInComplexifiedSpace} implies
that this intersection is a single complexified $K$-orbit.
Theorem~\ref{theorem_quotientStein} ensures that the quotient $Y\modmod K$ is a Stein space.
All $K^\C$-orbits
are closed, since $K$ and also $K^\C$
act freely such that all complex orbits have the same maximal dimension.\\
Next we show that each fiber of
$\pi_Y:Y\to Y\modmod K$ is the intersection of a $K^\C$-orbit with $Y$.
Let $K^\C\cdot x$ be a (closed) $K^\C$-orbit
for some $x\in Y$ and $F_Y=K^\C\cdot x\cap Y$
which is closed in $Y$. We will
show that $F_Y/K$ is connected. Suppose the contrary, then there is 
a function $f$ on $F_Y$ that is constant on each connected component
taking different values on each component. Of course, $f$ is holomorphic
as it is locally constant.
Since $Y$ is a Stein manifold there is a holomorphic function $\tilde f$ on $Y$
such that $\tilde f\vert_{F_Y}=f$. The set $Y$ is assumed to be a Runge subset, so for each
relatively compact subset $V\subset Y$ a holomorphic function $g$ on $X$ 
can be chosen which is 
arbitrarily close to $\tilde f$. Choosing $V$ stable by $K$ we can average over $K$, such that $g$
can be assumed to be $K$-invariant. But since $F_Y\subset K^\C\cdot x\cap X$ 
and every $K$-invariant holomorphic function on $X$ is constant
on $K^\C\cdot x\cap X$ the function $g$ is constant on that set and locally
arbitrarily close to a locally constant function with different values providing a contradiction.
So, we proved that each intersection of $Y$ with a single $K^\C$-orbit
is contained in a single $\pi_Y$-fiber. Since two
different $K^\C$-orbits are separated by $K$-invariant holomorphic functions on $X$,
their intersection with $Y$ are separated as well.
This proves that each $\pi_Y$-fiber is the intersection
of a $K^\C$-orbit with $Y$ and hence
$\imath_{\modmod K}:Y\modmod K\to X\modmod K$ is injective.\\
All fibers of $\pi$ and $\pi_Y$ have the same dimension and $X$, $Y$, $X\modmod K$ and $Y\modmod K$
are non-singular, so both projections are submersions.\\
In order to see the Runge property, choose some exhausting sequence $A_n\subset Y$ of compact subsets with
$A_n\subset A_{n+1}$ and $\bigcup A_n=Y$. Given a holomorphic function $f$ on $Y\modmod K$ and a
compact subset $D\subset Y\modmod K$. For some $N\in\N$ the set $D\subset\pi(A_N)$. The holomorphic function
$\pi^*f$ can be approximated by a global function on $X$ which can be averaged over $K$. This
averaging process does not destroy
the approximation property. The obtained function can be considered as a function on $X\modmod K$
approximating $f$ on $D$. This shows the Runge property.
\end{proof}

%%%%%%%%%%%%%%%%%%%%%%%%%%%%%%%%%%%%%%%%%%%%%
%
%   3   Local slice model
%
%%%%%%%%%%%%%%%%%%%%%%%%%%%%%%%%%%%%%%%%%%%%%

\section{Local slice model}\label{sect_localSliceModel}
After introducing the relevant objects and recalling the situation of a
compact group acting, we will now consider proper
actions of Lie groups by
holomorphic transformations more closely. Our goal in this section is to construct
local models for proper actions around totally real $G$-orbits.
\subsection{Definition of the local slice model}
Let $K$ be a compact subgroup of a Lie group $G$ and $\gtube$ a $G$-tube of $G$.
Theorem~\ref{theorem_gtubeExistence} shows that we can choose $\gtube$
to be $K$-stable, i.e. the $(G\times K)$-action
\begin{equation*}
\begin{array}{rccc}
\alpha_G:& (G\times K)\times G&\to&G\\
&(g,k,x)&\mapsto&(g\cdot x\cdot k^{-1})
\end{array}
\end{equation*}
extends to a $(G\times K)$-action $\alpha_\gtube$ by holomorphic transformations on $\gtube$.\\
Now let $V$ be a complex $K$-representation with (analytic) Hilbert quotient 
$\pi_V:V\to V\modmod K$ and
$S\subset V$ a $K$-stable neighborhood of the origin.
With the given $K$-action on $S$
we can define a
$(G\times K)$-action on the product $\gtube\times S$ via
\begin{equation*}
\begin{array}{rccc}
\alpha_{\gtube\times S}:& (G\times K)\times (\gtube\times S)&\to&\gtube\times S\\
&(g,k,\gamma,s)&\mapsto&(g\cdot \gamma\cdot k^{-1},k\cdot s)
\end{array}
\end{equation*}
This action induces a $G$-action on the formal Hilbert quotient
$$\gtube\times^{\modmod K}S=(\gtube\times S)\modmod K$$
in the following way.
Since the $G$- and the $K$-action commute, the induced
$(G\times K)$- and hence the $G$-action on $\osheaf(\gtube\times S)$ stabilizes the
subalgebra of $K$-invariant function $\osheaf^K(\gtube\times S)$. Thus the $(G\times K)$-action
pushes down to the formal Hilbert quotient, with $K$ contained in the ineffectivity.
Denoting by $[\gamma,s]=\pi(\gamma,s)$ a point in the
image of the quotient map 
$\pi:\gtube\times S\to \gtube\timescat{K}S$ we may write this as
\begin{equation*}
\begin{array}{rccc}
\alpha:& G\times (\gtube\timescat{K} S)&\to&\gtube\timescat{K} S\\
&(g,[\gamma,s])&\mapsto&[g\cdot \gamma, s]
\end{array}
\end{equation*}
Let us further assume $S$ to be a Stein manifold. On a Stein manifold
an action of a compact group is known to
provide an analytic Hilbert quotient which is a complex space,
even more, Theorem~\ref{theorem_quotientStein} shows
that the quotient is a Stein space. Since in our case the $K$-action
is free, the quotient $\gtube\timescat{K}S$ is non-singular
(Proposition~\ref{prop_shrinkingUnderFreeAction}),
hence a Stein manifold. Finally we will ask $S$ to be Runge in $V$.
This implies that the inclusion $S\to V$
induces an inclusion $S\modmod K\to V\modmod K$ 
(Proposition~\ref{prop_shrinkingUnderFreeAction}\eqref{inProp_openEmbedding})
and will
provide a shrinking property later (Theorem~\ref{theorem_shrinking} and Corollary~\ref{corollary_shrinkingKeepOpennessOnGQuotient}).
Finally we may suppose $S$ to be contractible in order to avoid cohomological obstructions.
Choosing a $K$-invariant scalar product on $V$ and set $\rho=\Vert . \Vert^2$ with respect to that scalar product,
all sets $S=\{v\in V\ \vert\ \rho(v)<c\}$ for $c>0$ are contractible Stein Runge neighborhoods of the origin, forming
a neighborhood basis.\\
We will use quotients of this type
as local models, precisely:
\begin{definition}
Let $\gtube$, $K$ and $S$ be as above. An open
subset $Y\subset X$
is called a
{\em local slice model} of $x$
if there is a
$G$-equivariant biholomorphic map
$\psi:\gtube\timescat{K}S\to Y$ mapping $[e,0]$ to $x\in X$.\label{defn_localSliceModel}
\end{definition}
For simplicity we sometimes identify $S$ with its image in $Y$. Note that $Y$ is automatically $G$-stable.
\subsection{Complexified orbits in slice, product and quotient}
In the subsequent section~\ref{subsection_quotientsOfLocalSliceModel} we will show that for a local slice
model $Y=\gtube\timescat{K}S$ the ringed spaces
$(Y\modmod G, \mathcal{O}_{Y\modmod G})$ and $(S\modmod K, \mathcal{O}_{S\modmod K})$ are isomorphic.
Since $S\modmod K$ is a complex space, this will be the way to endow $Y\modmod G$ with a
natural complex structure.\\
In this section, we will prepare this result in showing a one-to-one correspondence between the
complexified $K$-orbits in $S$ with the complexified $(G\times K)$-orbits
in the product $\gtube\times S$ and the complexified $G$-orbits in the
local slice model $\gtube\timescat{K}S$, denoted $\mathcal{B}_S^K$,
$\mathcal{B}_{\gtube\times S}^{G\times K}$ and $\mathcal{B}_{\gtube\timescat{K}S}^G$ respectively.\\
We consider the projection
$\kappa:\gtube\times S\to S$ onto the second component, the quotient map
$\pi:\gtube\times S\to \gtube\timescat{K}S$ and their induced maps
$\kappa^{-1}:\mathcal{P}(S)\to\mathcal{P}(\gtube\times S)$ and
$\pi^{-1}:\mathcal{P}(\gtube\timescat{K}S)\to\mathcal{P}(\gtube\times S)$ 
on the set of subsets.
\begin{proposition}
The maps $\kappa^{-1}$ and $\pi^{-1}$ induce bijections
$\kappa^{-1}: \mathcal{B}_S^K\to \mathcal{B}_{\gtube\times S}^{G\times K}$ and
$\pi^{-1}:\mathcal{B}_{\gtube\timescat{K}S}^G\to\mathcal{B}_{\gtube\times S}^{G\times K}$
between the corresponding complexified orbits. Furthermore
both bijections define bijections on the level of closed complexified orbits.
\label{proposition_bijectionOfOrbits}
\end{proposition}
For the proof we start with a technical lemma.
\begin{lemma}
Let $p:X\to Y$ be a continuous, open, surjective map between locally connected spaces and let each
fiber be connected. Then $p^{-1}:\mathcal{P}(Y)\to\mathcal{P}(X)$
induces a bijection between the connected components.\label{lemma_connectedComponents}
\end{lemma}
\begin{proof}
A subset of a locally connected space is a connected component if and only if it is closed, open and connected.
Writing $Y$ as the disjoint union of its connected components, $Y=\dot\bigcup C_\alpha$, then
$X=\dot\bigcup D_\alpha$ with $D_\alpha=p^{-1}(C_\alpha)$. The set $D_\alpha$ is open and due to
$D_\alpha=X\backslash\bigcup_{\beta\not=\alpha}D_\beta$ it is also closed.
Now let $E$ be a connected component of $X$ hence contained in some set $D_\alpha$.
For each point $d\in E$ the fiber $p^{-1}(p(d))$ is connected by assumption hence
entirely contained in $E$, thus $E$ is $p$-saturated, i.e. $E=p^{-1}(p(E))$.
Both set $p(E)$ and $p(D_\alpha\backslash E)$ are open by assumption of $p$ being open and since they
are complementary in $C_\alpha$ also closed. Thus one of them must be empty, hence $D_\alpha=E$.
We showed that the map $p^{-1}$ maps connected components to connected
components. It is surjective, since
each point is contained in some set $D_\alpha$. Finally, $p^{-1}$ is injective,
since $p$ is assumed to be surjective.
\end{proof}
\begin{proof}[Proof of Proposition~\ref{proposition_bijectionOfOrbits}]
We will apply Lemma~\ref{lemma_connectedComponents} to the maps
$\kappa:\gtube\times S\to S$ and $\pi:\gtube\times S\to\gtube\timescat{K}S$.
In order to ensure continuity and openness of $p$ and $\pi$ with respect to the $\tau$-topologies,
we will use
Lemma~\ref{lemma_topologyCompatibleEquivariantHolomorphism} which shows that a $G$-equivariant holomorphic map is
$\tau_G$-continuous and $\tau_G$-open. For this, we define a $(G\times K)$-action on $S$
by extending the $K$-action ineffectively to $G$. The induced topologies $\tau_K$ and
$\tau_{G\times K}$ are identical. Thus Lemma~\ref{lemma_topologyCompatibleEquivariantHolomorphism}
shows that $\kappa$ is
$\tau_{G\times K}$-$\tau_K$-continuous and -open. For the projection $\pi$ in turn, observe that the
$G$-action on the local slice model $\gtube\timescat{K}S$ is defined 
as induced by the $(G\times K)$-action on $\gtube\times S$ pushed down to the quotient
with $\Delta=\{(k^{-1}, k)\vert k\in K\}\subset G\times K$ contained in the inefficiency.
Thus the topologies $\tau_{G\times K}$ and $\tau_G$ are identical. Analogously
Lemma~\ref{lemma_topologyCompatibleEquivariantHolomorphism}
provides that $\pi$ is $\tau_{G\times K}$-$\tau_G$-continuous and -open.
In order to be able to apply Lemma~\ref{lemma_connectedComponents} we are left to show
that the fibers of $\kappa$ and $\pi$ are $\tau_{G\times K}$-connected. The $G$-action and the $K$-action on
the product $\gtube\times S$ induce finer topologies $\tau_G$ and $\tau_K$, i.e.
$\tau_G\subset \tau_{G\times K}$ and $\tau_K\subset\tau_{G\times K}$. 
A $\kappa$-fiber $\gtube\times \{s\}$ is a complexified $G$-orbit (Lemma~\ref{lemma_openSetsInTube}),
hence $\tau_G$-connected and therefore connected with respect to the stronger topology $\tau_{G\times K}$.
Proposition~\ref{prop_shrinkingUnderFreeAction}\eqref{inProp_singleComplKorbit}
shows that each $\pi$-fiber is a complexified $K$-orbit, hence $\tau_K$-connected, and therefore
connected with respect
to the stronger topology $\tau_{G\times K}$.
Thus we are allowed to apply Lemma~\ref{lemma_connectedComponents} to both maps $\kappa$ and $\pi$ 
and obtain the desired bijections
$\kappa^{-1}: \mathcal{B}_S^K\to \mathcal{B}_{\gtube\times S}^{G\times K}$ and
$\pi^{-1}:\mathcal{B}_{\gtube\timescat{K}S}^G\to\mathcal{B}_{\gtube\times S}^{G\times K}$.\\
Now we turn back to the classical topologies.
Both projections $\pi$ and $\kappa$ are continuous and open (Proposition~\ref{prop_shrinkingUnderFreeAction}\eqref{inProp_projectionOpen}). Thus
both induce a bijection between the saturated closed sets and closed set in the target
and in particular inverse images of closed complexified orbits are closed complexified orbits..
\end{proof}\label{subsection_sliceProductQuotient}
\subsection{Quotients of the local slice model}
In this section we can give $Y\modmod G$ the structure
of a complex space. The sheaf $\osheaf_{Y\modmod G}$ is defined as $\pi_*\osheaf^G_Y$, i.e.
given an open subset $U\subset Y\modmod G$ then $f\in\osheaf_{Y\modmod G}(U)$
if $\pi^*f$ is a ($G$-invariant) holomorphic function on $\pi^{-1}(U)$.
\begin{theorem}
Let $Y$ be a local slice model with slice $S\subset Y$ and denote
$\pi_S:S\to S\modmod K$ the analytic Hilbert quotient and $\pi:Y\to Y\modmod G$ the formal Hilbert quotient. Then 
the embedding $\imath_S:S\to Y$ induces an isomorphism
$(S\modmod K,\osheaf_{S\modmod K})\to(Y\modmod G, \osheaf_{Y\modmod G})$.\label{theorem_SmodmodKIsomorphicToYmodmodG}
\end{theorem}
Besides the result on the correspondence of complexified $K$-, $(G\times K)$- and $G$-orbits in
$S$, $\gtube\times S$ and $Y$ from Section~\ref{subsection_sliceProductQuotient}
above we need some results
stating similar 
geometrical properties for the fibers of $\pi:Y\to Y\modmod G$
analogous as described in Theorem~\ref{theorem_fiberStructureAHQcptCase}.
\begin{proposition}
Let $Y$ be a local slice model with slice $S\subset Y$ and denote
$\pi_S:S\to S\modmod K$ the analytic Hilbert quotient and $\pi:Y\to Y\modmod G$ the formal Hilbert quotient. Then each
complexified $G$-orbit in $Y$ corresponds to a complexified $K$-orbit in $S$. Furthermore, for every $p\in Y\modmod G$
and every $x\in\pi^{-1}(p)$
\begin{enumerate}
\item the closure of $\overline{B_Y(x)}$ is the union of complexified $G$-orbits $B_Y(y)$ with
$\dim B_Y(y)< \dim B_Y(x)$ for each $y\not\in B_Y(x)$.\label{item_localSliceModel_dimensionInClosureStrictlySmaller}
\item there is exactly one complexified $G$-orbit of lowest dimension, denoted by $E_Y(p)$. This complexified
$G$-orbit is closed.\label{item_localSliceModel_oneSingleComplexifiedOrbit_lowestDimension_closed}
\item $E_Y(p)$ lies in the closure of each complexified $G$-orbit in the fiber, formally
$E_Y(p)\subset\overline{B_Y(x)}$.
\end{enumerate}
The embedding $\imath_S:S\to Y$ induces a homeomorphism
$\varphi_S:S \modmod K\to Y\modmod G$
in the sense that the following diagram commutes.
\begin{eqnarray*}
S\quad \ &\stackrel{\imath_S}{\hookrightarrow}& Y\notag\\
\downarrow \pi_S&\quad&\downarrow\pi\\
S\modmod K&\stackrel{\varphi_S}{\longrightarrow}& Y\modmod G\notag
\end{eqnarray*}
\label{proposition_fiberStructureInLSM}
\end{proposition}
\begin{proof}
 For $x_0=[\gamma_0,s_0]\in Y$ the complexified $G$-orbit $B_Y(x_0)$ is the
 image of $i_{s_0}:\gtube\to Y, \gamma\mapsto [\gamma, s_0]$.
 For each $s_0\in S$ the complexified $K$-orbit $B_S(s_0)$ is
 the intersection $B_Y([e,s_0])\cap S$. This realizes the one-to-one corresponence between
 complexified $K$-orbits in $S$ and complexified
 $G$-orbits in $Y$. $B_S(s_0)$ is closed if and
 only if $B_Y([e,s_0])$ is closed. Since each $\pi_S$-fiber contains a unique
 complexified $K$-orbit (Theorem~\ref{theorem_fiberStructureAHQcptCase}), the same is true for $\pi$-fibers. 
 This implies that the results of Theorem~\ref{theorem_fiberStructureAHQcptCase} applied to $S$ transform
 to the corresponding results on $Y$. Consequently, the map $S\modmod K\to Y\modmod G$
 represents the identification of these closed complexified orbits,
 imposing a bijection which is a continuous map by construction. In order to establish that this identification is a homeomorphism we will show that it is a proper map. For this, let $s_n\in S$ be a sequence.
The induced sequence $\pi_S(s_n)\in S\modmod K$
is mapped via the considered identification to
$\pi([e,s_n])\in Y\modmod G$. Assume that this sequence converges and we aim to show that then
$\pi_S(s_n)$ converges. Every $G$-invariant holomorphic function $f$ on $Y$ induces a continuous function $h$ on $Y\modmod G$. As $h(\pi([e,s_n]))$ converges for every such function $h$, the sequence $f([e,s_n])=h(\pi([e,s_n]))$ converges for every $G$-invariant
holomorphic function $f$ on $Y$. As any $K$-invariant holomorphic map $f_S$ on $S$ can be interpreted as a $G$-invariant
holomorphic map $f$ on $Y$, we deduce that
$f_S(s_n)$ converges for every $K$-invariant holomorphic function. As any holomorphic function on the quotient $S\modmod K$ can be seen as a $K$-invariant
holomorphic function on $S$, the image of the sequence
$\pi_S(s_n)$ under any holomorphic function on $S\modmod K$ converges. Finally, since $S\modmod K$ is a Stein space, this imposes that $\pi_S(s_n)$ itself converges.\\
This shows that $S\modmod K\to Y\modmod G$
is a homeomorphism.
\end{proof}
\begin{proof}[Proof of Theorem~\ref{theorem_SmodmodKIsomorphicToYmodmodG}]
Induced from the inclusion
$\imath_S:S\hookrightarrow Y$ there is a homeomorphism 
$\varphi_S:S\modmod K\to Y\modmod G$
(Proposition~\ref{proposition_fiberStructureInLSM}).
Thus for an open subset $U\subset Y\modmod G$
the intersection $W_S$ of $W=\pi^{-1}(U)$
with $S$ equals $\pi_S^{-1}(\varphi_S^{-1}(U))$.
Hence $V_S$ is $\pi_S$-saturated.
Given a holomorphic function
$f:U\to\C$ the pullback
$\pi^*f$ is $G$-invariant and as $V_S$ is $K$-stable the function's restriction
$\pi^*f\vert_{V_S}$ is a $K$-invariant
holomorphic function
$g:\varphi_S^{-1}(U)\to\C$.
This defines a map
$\mathcal{O}_{Y\modmod G}(U)\to\mathcal{O}_{S\modmod K}(\varphi_S^{-1}(U))$.\\
Now we want to define
a map $\mathcal{O}_{S\modmod K}(V)\to
\mathcal{O}_{Y\modmod G}(\varphi_S(V))$
for any open subset $V\subset S\modmod K$.
On the level of (open) sets the identifications
in Proposition~\ref{proposition_bijectionOfOrbits}
show that
$\gtube\timescat{K}\pi_S^{-1}(V)
=\pi^{-1}(\varphi_S(V))$.
A holomorphic function $g:V\to\C$
arises by definition from a $K$-invariant 
holomorphic 
function $\pi_S^*g:\pi_S^{-1}(V)\to\C$.
We define
$\hat h:\gtube\times\pi_S^{-1}(V)\to\C$ by
$\hat h(\gamma, s)=\pi_S^*g(s)$. It is a
$(G\times K)$-invariant holomorphic
function which pushes down to the
$G$-invariant holomorphic function
$h:\gtube\timescat{K}\pi_S^{-1}(V)\to\C$.
This identification defines the desired 
inverse.
\end{proof}
We showed that for every local slice model
$Y$ the quotient $\pi:Y\to Y\modmod G$
is an analytic Hilbert quotient.
\label{subsection_quotientsOfLocalSliceModel}
\subsection{Shrinking of the local slice model}
In our proof of
the existence of local slice models $\gtube\times^{\modmod K}S$ around
special points (Section~\ref{subsect_existenceLocalSliceModel}) and
later
for the existence of invariant
local potentials around these points (Section~\ref{sect_localPotential})
we will need to shrink $\gtube\times^{\modmod K}S$ to an arbitrarily small local slice model $T\times^{\modmod K}S_0$.
Herein $T\subset \gtube$ is a smaller $G$-tube as provided by definition of $\gtube$
and $S_0\subset S$ is a smaller slice with $S_0$ Runge in $S$ and $S_0$ Stein and $K$-stable.
\begin{theorem}\textbf{\em(Shrinking Theorem of Local Slice Model)}\label{theorem_shrinking}
Let the set $Y=\gtube\times^{\modmod K}S\subset X$ be a local slice model around $x$ and $U\subset Y$ a
$G$-stable neighborhood of $x$.
Then there is a $G$-subtube $T\subset \gtube$
and an open $K$-stable Stein neighborhood $S_0$ of the origin, Runge in $S$, such that
$Y_0=T\times^{\modmod K}S_0$ is a local
slice model around $x$.
\end{theorem}
\begin{corollary}
The open embedding $\imath_0:Y_0\hookrightarrow Y$
of the shrunk local slice model in Theorem~\ref{theorem_shrinking} induces an
open embedding $\bar\imath_0:Y_0\modmod G\hookrightarrow Y\modmod G$.
\label{corollary_shrinkingKeepOpennessOnGQuotient}
\end{corollary}
\begin{proof}[Proof of Theorem~\ref{theorem_shrinking}
and Corollary~\ref{corollary_shrinkingKeepOpennessOnGQuotient}]
Let $p:\gtube\times S\to \gtube\times^{\modmod K}S$ denote the projection
of the analytic Hilbert quotient 
and set $\hat U=p^{-1}(U)$. There are open 
subsets $A\subset\gtube$ and $B\subset S$ such that $(e,0)\in A\times B\subset \hat U$.
The definition of $G$-tubes
provides
a $G$-tube $T\subset A$ that is Runge in $\gtube$.
The manifold
$S$ can be realized as a $K$-stable open neighborhood of the origin in a unitary $K$-representation $V$. 
Let $\rho_V$ be the squared norm function of the
$K$-invariant norm on $V$. For a
sufficiently small $\delta>0$, the set
$S_0=\{x\in V\ \vert\ \rho_V(x)<\delta\}\subset S$ is a Stein Runge subset of $S$.
Since the diagonal $K$-action on $\gtube\times S$ is free and $T\times S_0$ is an open Runge Stein subset
of $\gtube\times S$, Proposition~\ref{prop_shrinkingUnderFreeAction} implies
that $Y_0=T\times^{\modmod K}S_0$ can be seen as an open Runge Stein subset of $Y=\gtube\times^{\modmod K}S$.
Part~\eqref{inProp_openEmbedding} of this proposition
shows that the inclusion $S_0\hookrightarrow S$
induces an open embedding
$S_0\modmod K\hookrightarrow S\modmod K$.
From
Theorem~\ref{theorem_SmodmodKIsomorphicToYmodmodG} we obtain canonical biholomorphisms
$S_0\modmod K\cong Y_0\modmod G$ and
$S\modmod K\cong Y\modmod G$. Combining these facts
we see that
the inclusion $Y_0\hookrightarrow Y$
induces a holomorphic
open embedding
$Y_0\modmod G\to Y\modmod G$.
\end{proof}
\subsection{Existence of the local slice model}\label{subsect_existenceLocalSliceModel}
As a next step we construct a local slice at a
complexified $G$-orbit through
$x\in X$ if the orbit $G\cdot x$ is a totally
real submanifold of $X$. The construction
will give a local slice model of $x$
for the $G$-action on $X$.
\begin{theorem}\textbf{\em (Local Slice Theorem)}
Let $X$ be a Stein manifold and $G$ a Lie group acting properly by holomorphic transformations.
If $G\cdot x$ is a totally real submanifold of $X$, there is
a local slice model $Y$ at $x$.\label{theorem_slice}
\end{theorem}
Let $K$ denote the compact isotropy group at $x$. The slice $S\subset Y$ can be chosen $K$-stable such that $S$ is $K$-equivariantly biholomorphic to an open neighborhood
of the origin of a $K$-stable complex linear subspace in $T_xX$. This subspace is in fact complementary to
the complex subspace generated by
$T_x(G\cdot x)\subset T_xX$.
Note that the complex subspace generated by
$T_x(G\cdot x)$ is the right candidate for the
tangent space of the complexified $G$-orbit
through $x$.\\
We first show the following technical lemma.
\begin{lemma}
Let $G$ act properly on $X_2$ and 
let $\varphi:X_1\to X_2$ be a $G$-equivariant local diffeomorphism with $\varphi(x_1)=x_2$ and suppose that
$\varphi\vert_{G\cdot x_1}$ is injective.
Then there are $G$-stable
neighborhoods $U_i$ around $x_i$ such that $\varphi\vert_{U_1}:U_1\to U_2$ is a
diffeomorphism.\label{lem_diffeoAroundOrbit}
\end{lemma}
\begin{proof}
First we recall that the action
on $X_1$ is proper as well.
Suppose that in each neighborhood of the orbit $G\cdot x_1$ there is a pair of points $u\not=v$ mapping to the same point. So we
get sequences $u_n, v_n$ with $\varphi(u_n)=v_n$ and $u_n\not= v_n$. Using the action, we may assume that $u_n$
converges to $u\in G\cdot x_1$. There is a sequence $g_n\in G$ such that $\tilde v_n=g_n\cdot v_n$ converges to a 
point in $G\cdot x_1$. This can even be arranged such that $\tilde v_n$ converges to $u$ likewise.
Thus we get two sequences both converging to $f(u)$, namely $f(\tilde v_n)$ and
$g_n^{-1}f(\tilde v_n)=f(v_n)=f(u_n)$. Properness of the action implies the existence of
a convergent subsequence of $g_n$, which for simplicity we also denote $g_n$.
The sequence $f(\tilde v_n)$ converges 
therefore
to $f(u)$ as well as to $gf(u)=f(gu)$, hence $f(u)=f(gu)$. Injectivity on the orbit $G\cdot x_1$
provides $u=gu$. So, $u$ is the limit of both $u_n$ and $v_n$. Around $u$, the map $f$ is a local diffeomeorphism, so $u_n=v_n$
for large $n$, contradicting the assumption.
\end{proof}
\begin{proof}[Proof of Theorem~\ref{theorem_slice}]
We have to prove the existence of a local slice
model at $x$. Let $K=G_x$ denote the isotropy group at $x$ which is compact since $G$ acts properly
on $X$. On a complex manifold an action of a compact
group is linearizable in a neighborhood of a
fixed point. This implies that $x$ has an open $K$-stable neighborhood $U$ of $x$
which is $K$-equivariantly biholomorphic
to an open neighborhood $W$ of $0$ in $T_xX$
where $K$ acts on $T_xX$ by the isotropy 
representation of $K$ at $x$.
The complex span of $T_x(G\cdot x)\subset T_xX$
is a $K$-invariant subrepresentation
of $T_xX$. We choose a $K$-invariant 
complex complement $V$ and set
$S_W=V\cap W$.\\
We use the linearization map to
identify $S_W$ with a $K$-stable 
local complex submanifold $S$ of $X$ with $x\in S$.
The restriction 
$\alpha:G\times S\to X, (g,y)\mapsto g\cdot y$
of the action map is real analytic and for 
fixed $g\in G$ the map
$\alpha\vert_{\{g\}\times S}:\{g\}\times S\to X$ is holomorphic.
It follows that $\alpha$ extends to a holomorphic map $\alpha^c:\Omega\to X$ where $\Omega\subset \gtube\times S$ is an open neighborhood
of $G\times\{x\}$ in $\gtube\times S$.
The map $\alpha^c$ is locally $G$-equivariant
and $\gtube$ is $G$-equivariantly diffeomorphic to
$G\times \Sigma$ by definition. This implies that
$\alpha^c$ extends $G$-equivariantly and holomorphically to a $G$-stable
open neighborhood, i.e. after shrinking of $\gtube$
and $S$ we may assume $\Omega=\gtube\times S$.
The map $\alpha^c:\gtube\times S\to X$
is $K$-invariant where $K$
acts on $\gtube \times S$ by $k\cdot(\gamma,y)=
(\gamma\cdot k^{-1},k\cdot y)$ and after shrinking
we may assume that it is a submersion onto its
open image.\\
Since $\alpha^c$ is $K$-invariant, it pushes down to a holomorphic
map
\begin{align*}
 \varphi:\gtube\timescat{K} S&\to X
\end{align*}
submersive
to its image. Denote $B=\alpha^c(\gtube\times\{0\})
=\varphi(\gtube\timescat{K}\{0\})$.
Since the $G$-orbit is totally real, 
\begin{equation*}
\dim \gtube\modmod K=\dim_\R G/K=\dim B \ .
\end{equation*}
Thus a calculation of the dimension
\begin{equation*}
 \dim \left(\gtube\timescat{K} S\right)=\dim \left(\gtube\modmod K\right)+\dim S=\dim B+\dim S=\dim X
\end{equation*}
shows that $\varphi$ is in fact a $G$-equivariant local diffeomorphism.
Lemma~\ref{lem_diffeoAroundOrbit} in turn shows that this becomes a diffeomeorphism to the image after restricting $\varphi$
to some $G$-stable neighborhood of $[e,0]$. In a last step, using the Shrinking Theorem of Local Slice Models
(Theorem~\ref{theorem_shrinking}), this can be
restricted to a neighborhood of the form
$\gtube\timescat{K} S$ by sufficient shrinking of $\gtube$ and $S$.
\end{proof}

%%%%%%%%%%%%%%%%%%%%%%%%%%%%%%%%%%%%%%%%%%%%%
%
%   4   Local potentials
%
%%%%%%%%%%%%%%%%%%%%%%%%%%%%%%%%%%%%%%%%%%%%%

\section{Local potentials}\label{sect_localPotential}
The construction of both the complex structure (Section~\ref{section_cplxStructure}) and in particular the K\"ahler structure
(Section~\ref{section_KaehlerStructure}) make
use of local potentials.\\
The main purpose of this section is to prove Theorem~\ref{theorem_localPotentialAroundM} stating the existence of a local potential
around each point in the momentum zero
level $\mathcal{M}$.\\
Additionally, the construction of a local potential gives rise to the question whether away
from the momentum zero level local potentials do always exist as well.
For this, we have to impose further natural conditions, namely the $G$-orbit of
the considered point has to be totally real and the K\"ahler form restricted to the orbit has to be exact.\\
The $G$-orbits in the momentum zero level are isotropic and therefore the orbits are perforce
totally real and $\omega$ restricted to the orbit is exact.
So, on $\mathcal{M}=\mu^{-1}(0)$ these conditions are fulfilled automatically.
\subsection{Definition of a local potential}
\begin{definition}\label{definition_potential}\mbox{ }
\begin{enumerate}
\item  Given a Hamiltonian K\"ahler  manifold $(X,G,\omega,\mu)$ a {\em potential} is a
smooth $G$-in\-var\-iant strictly plurisubharmonic function $\rho$ such that
$\omega=\de\dc\rho$ and for each $v\in\lie{G}$ the component of the momentum map satisfies $\mu^v=-\imath_{\tilde v}\dc\rho$.
Here $\tilde v$ is the vector field of $v$ induced by the $G$-action.
\item A function is called a {\em local potential} around $x\in X$ if it is
a potential on some $G$-stable neighborhood of $x$.\label{defn_localPotential}
\end{enumerate}
\end{definition}
Observe that this makes sense as we obtain the momentum map condition
\begin{equation*}
\de\mu^v=-\de\imath_{\tilde v}\dc\rho=\imath_{\tilde v}\de\dc\rho=\imath_{\tilde v}\omega
\end{equation*}
from $G$-invariance of $\rho$ and $\dc\rho$ since invariance reads in formula
\begin{equation*}
0=\lieder_{\tilde v}\dc\rho=  \imath_{\tilde v}\de\dc\rho + \de\imath_{\tilde v}\dc\rho\ .
\end{equation*}
Vice versa, given a smooth $G$-in\-var\-iant strictly plurisubharmonic function $\rho$, it defines a Hamiltonian structure, i.e.
an invariant K\"ahler form $\omega=\de\dc\rho$
and an equivariant momentum map defined by
$\mu^v=-\imath_{\tilde v}\dc\rho$.\\
We obtain strong relations between the momentum zero level $\mathcal{M}$ and the
analytic Hilbert quotient if
we impose our exhaustion property of the potential
(Definition~\ref{definition_exhaustion}).
\begin{remark}
On Stein manifolds an exhaustive potential $\rho$ always exists \cite{Gra58}.
By averaging over a compact group $K$
it can be made $K$-invariant.
\end{remark}
The moment zero level $\mathcal{M}$ is $G$-stable by construction. Therefore the inclusion
$i:\mathcal{M}\to X$ induces a map $i_G:\mathcal{M}/G\to X\modmod G$ to the formal Hilbert quotient.\\
In the situation of a compact group we have the following theorem.
\begin{theorem} Let the compact Lie group $K$ act on the Stein
manifold $X$ by holomorphic transformations and let a $K$-exhaustive potential $\rho$
induce a Hamiltonian structure on $X$. Then the map $i_K:\mathcal{M}/K\to X\modmod K$
 is a homeomorphism. In particular, for all $p\in X\modmod K$ the set $\mathcal{M}_p=\mathcal{M}\cap \pi^{-1}(p)$ is not empty and
$\mathcal{M}_p\subset E(p)$.
\end{theorem}
For the proof see \cite{HeiKut95}.
There it applies only
for 
proper potentials unbounded from above. An exhaustive
potential $\rho$ in the sense of this paper can be made unbounded from above keeping properness
by choosing a
suitable strictly convex diffeomorphism $\chi\nolinebreak :\nolinebreak (-\infty,\sup\rho)\to\R$ and using $\tilde\rho=\chi\circ\rho$ as
exhaustive potential. This
does not change the momentum zero level $\mathcal{M}$ (compare Lemma~\ref{lemma_bendingUpPotential}).
\begin{corollary}
Let $X$ be a Stein $(G\times K)$-manifold, $K$
acts freely and $\rho$ is a strictly
plurisubharmonic $(G\times K)$-exhaustive
potential inducing
a K\"ahler form and a momentum
map $\mu_K:X\to \lie{K}^*$. Then $\mathcal{M}=\mu_K^{-1}(0)$,
$\mathcal{M}/K$ and $X\modmod K$ are 
manifolds with induced $G$-action,
the $G$-equivariant
inclusion $\mathcal{M}\hookrightarrow X$ induces a $G$-equivariant homeomorphism
and $\rho\vert_\mathcal{M}$ pushes down to
a strictly plurisubharmonic $G$-exhaustion on $X\modmod K$.
\label{corollary_GexhaustionPushesDowntoMmodKforFreeKaction}
\end{corollary}
\begin{proof}
Proposition~\ref{prop_shrinkingUnderFreeAction}\eqref{inProp_nonSingular} shows
that $X\modmod K$ is a manifold.
As the $K$-action is free the momentum map
$\mu_K$ is submersive thus $\mathcal{M}$ is a manifold. Since $K$ acts freely on $\mathcal{M}$
the quotient $\mathcal{M}/K$ is a manifold as well
and $\mathcal{M}\to\mathcal{M}/K$ is a $K$-principle
bundle. Since the $G$- and the $K$-action
commute, $\mathcal{M}$ is $G$-stable and 
both quotients $\mathcal{M}/K$ and $X\modmod K$
inherit a proper $G$-action such that the
quotient maps $p:\mathcal{M}\to\mathcal{M}/K$
and $\pi:X\to X\modmod K$ are $G$-equivariant.
For $y\in X\modmod K$ the $\pi$-fiber
$F=\pi^{-1}(y)$ is a Stein manifold.
Since the isotropy of the $G$-action at $y\in X\modmod K$ is compact, the restriction $\rho\vert_F$
is a $K$-exhaustion. Thus \cite{HeiKut95}
shows that $\mathcal{M}$ intersects $F$ exactly in a single $K$-orbit. Hence the $G$-equivariant inclusion
$i:\mathcal{M}\to X$ induces a $G$-equivariant bijection
$\varphi:\mathcal{M}/K\to X\modmod K$.
By the same arguments as in \cite{HeiKut95} the map
$\varphi_{/G}:(\mathcal{M}/K)/G\to (X\modmod K)/G$
is a homeomorphism and since the $G$-action on 
$\mathcal{M}/K$ and $X\modmod K$
is proper the bijection $\varphi$ is a
homeomorphism as well.\\
The restricted function $\rho\vert_{\mathcal{M}}$
is a $K$-invariant $G$-exhaustion
and induces therefore a $G$-exhaustion
on the quotient $\mathcal{M}/K$.
For every point $y\in\mathcal{M}/K$
the principle bundle
$p:\mathcal{M}\to\mathcal{M}/K$ possesses a local section $\sigma$. Locally the pushed down function can be written as $\sigma^*\rho$ independently of the choice of $\sigma$. For the point $x=\sigma(y)$ the covector
$\dc\rho\vert_{T_x\mathcal{M}}$ has $T_x(G\cdot x)=\ker (p_*)_x$
contained in its kernel. Thus
$\sigma^*\dc\rho=\dc\sigma^*\rho$.
Further $T_x(G\cdot x)$ is the kernel of
$\omega\vert_{T_x\mathcal{M}}$. Hence
$\de\dc\sigma^*\rho=\sigma^*\omega$ is
a K\"ahler form and $\sigma^*\rho$ is strictly plurisubharmonic.
\end{proof}
\subsection{Potential on complexified $G$-orbit}
In this section we will take a look at the
$G$-invariant potential on a complexified $G$-orbit.
They turn out to be minimal and critical on ``isolated'' $G$-orbit.
\begin{definition}
Given a smooth $G$-invariant function $\rho$.
We say that $\rho$ has an {\em local isolated minimum locus}
at $p$ if there is a $G$-stable neighborhood $U$
of $p$
such that $\rho(x)>\rho(p)$ for all $x\in U\backslash (G\cdot p)$.\label{definition_isolatedMinimum}
\end{definition}
We get the following local version on
a complexified $G$-orbit.
\begin{lemma}
Let $B$ be a complexified $G$-orbit, $\rho$ a $G$-invariant strictly plurisubharmonic function and $p_0\in B$ a critical point of $\rho$. If
$\Sigma_B$ is a local submanifold transversal to the $G$-orbit $G\cdot p_0$ at $p_0$ of maximal dimension, then the restriction $\rho\vert_{\Sigma_B}$ has positive Hessian at $p_0$ and the $G$-orbit is a local isolated minimum locus
of $\rho$.\label{lemma_positiveHessianTransversalToCriticalOrbitInComplexifiedGorbit}
\end{lemma}
\begin{proof}
Denote $\tilde v$ the fundamental vector field of $v\in\lie{G}$. Fix a
linear subspace $W\subset \lie{G}$ such that the linear map $W\to T_{p_0}(G\cdot p_0), v\mapsto \tilde v_{p_0}$ is a bijection. For each $v\in W$, $v\not=0$, we calculate
\begin{eqnarray*}
 J\tilde v(J\tilde v(\rho))_{p_0}&=&
 -(\imath_{J\tilde v}\de\imath_{\tilde v}\dc\rho)_{p_0}\\
 &=&(\imath_{J\tilde v}\imath_{\tilde v}\de\dc\rho)_{p_0}\\
 &=&\omega_{p_0}(\tilde v, J\tilde v)>0\\
\end{eqnarray*}
since $\lieder_{\tilde v}\dc\rho=\dc\lieder_{\tilde v}\rho=0$ and $\omega_{p_0}(\tilde v, J\tilde v)>0$ for all $\tilde v_{p_0}\not =0$.
From this we conclude that $T_{p_0}(G\cdot p_0)$
and $JT_{p_0}(G\cdot p_0)$ intersect trivially, hence the orbit is totally real and it is of maximal dimension since
$T_{p_0}(G\cdot p_0)+JT_{p_0}(G\cdot p_0)$ span
the tangent space of the complexified orbit $B$ by definition.
Denote $\tilde v$ the induced vector field of $v\in \lie{G}$ and for a vector field $\xi$ the flow
for time $t$ by $\varphi^\xi_t(x)$ starting at point $x\in X$.
Then there is an open neighborhood $U$ of the origin in $W$
such that
\mbox{$U\to B, v\mapsto\varphi_1^{J\tilde v}(p_0)$} is
an immersion and its image $\Sigma$
is a local submanifold transversal to the $G$-orbit and of maximal dimension for which the above calculations shows
that the Hessian of $\rho$ is positive with respect
to local linear coordinates coming from the identification
$U\to \Sigma$. As $\rho$ is supposed to be critical a $p_0$, $\rho\vert_\Sigma$ has a local isolated minimum at $p_0$. 
\end{proof}
From the exhaustion property of a function $\rho$
we get a result on the existence of critical points.
\begin{lemma}
Let $G$ act properly on the manifold $X$, 
assume $X/G$ to be connected and let
$\rho:X\to\R$ be a smooth $G$-exhaustive
potential. If $\rho$ possesses two different
$G$-orbits $G\cdot p_1$ and $G\cdot p_2$ which
are local isolated minimum loci, then there is a critical
point $q\not\in G\cdot p_1\cup G\cdot p_2$
such that $q$ is not a minimum locus.
\label{lemma_saddlePoint}
\end{lemma}
\begin{proof}
First let us pass to the quotient, i.e. introduce
the quotient map $\beta:X\to Z=X/G$, set $\tilde p_i=\beta(p_i)$ and define $\tilde\rho$ by
$\tilde\rho\circ\beta=\rho$. Denote for each
$b\in\R$ the sets $X_b=\{x\in X\ \vert\ \rho(x)< b\}$ and $Z_b=\{z\in Z\ \vert\ \tilde\rho(z)<b\}$.
The function $\tilde\rho$
is exhaustive and $Z$ is connected, so there is a $C\in\R$ such that $\tilde p_1, \tilde p_2$ lie
in the same connected component of $Z_C$.\\
We may assume that
$\tilde\rho(\tilde p_1)\geq\tilde\rho(\tilde p_2)$
and claim that there is a 
$c>\tilde\rho(\tilde p_1)$ such that the points
$\tilde p_1$ and $\tilde p_2$ lie in different connected components of $Z_c$.
By assumption $\tilde p_1$ is a local isolated minimum point of $\tilde \rho$, i.e. there is an open neighborhood 
$V$ of $\tilde p_1$ such that
\begin{equation*}
\{v\in V\ \vert \ \tilde\rho(v)>\tilde\rho(\tilde p_1)\}=V\backslash\{\tilde p_1\}\ .
\end{equation*}
So there is a $c>\tilde\rho(\tilde p_1)$ such that
$U_c=\{v\in V\ \vert\ \tilde\rho(v)< c\}=Z_c\cap V$ is relatively compact in $V$. Thus $Z_c$ is
the disjoint union of $U_c$ and $Z_c\backslash V$, the sets are disconnected and $\tilde p_1\in U_c$
and $\tilde p_2\in Z_c\backslash V$. Let $\kappa$
be the supremum of all $c>\tilde\rho(\tilde p_1)$
such that the points
$\tilde p_1$ and $\tilde p_2$ lie in different connected components of $Z_c$.\\
Assume $\de\rho(q)\not =0$ for all $q\in\partial X_\kappa$.
As $\partial Z_\kappa$ is compact, there is a
$\hat\kappa>\kappa$ such that
$X_\kappa$ is a deformation retract of $X_{\hat\kappa}$, e.g. just by choosing the negative gradient
flow of $\rho$.
Additionally this deformation can be chosen
to be $G$-equivariant,
so that $Z_\kappa$ is a deformation retract of $Z_{\hat\kappa}$. Hence for each $i=1,2$ the connected component of $\tilde p_i$ in $Z_{\hat\kappa}$ is retracted to the connected component of $\tilde p_i$ in $Z_\kappa$. Thus the points the points
$\tilde p_1$ and $\tilde p_2$ still lie in different connected components of $Z_{\hat\kappa}$. This contradicts
the choice of $\kappa$ as a supremum with this property.
So, we may conclude that there is a point $q\in\partial X_\kappa$
with $\de\rho(q)=0$. As each neighborhood of $q$ intersects $X_\kappa$ the function $\rho$ does not become locally minimal at $q$.
\end{proof}
Combining the results of the previous 
two lemmata we obtain the following
global statement on complexified $G$-orbits
as a corollary.
\begin{corollary}
Let $B$ be a complexified $G$-orbit, $\rho:B\to\R$ a 
$G$-exhaustive strictly plurisubharmonic function and $p\in B$ a critical point of $\rho$. Then
the $G$-orbit through $p$ contains the only points
on which $\rho$ becomes critical or minimal.\label{corollary_uniqueCriticalOrbitInComplexifiedGorbitIfExhaustive}
\end{corollary}
\begin{proof}
Since $\rho$ is $G$-exhaustive, there is at least
one $G$-orbit on which $\rho$ becomes minimal and therefore critical. 
Assume there is a second $G$-orbit
on which $\rho$ becomes critical, then by Lemma~\ref{lemma_positiveHessianTransversalToCriticalOrbitInComplexifiedGorbit} both orbits are local
isolated minimum loci and
Lemma~\ref{lemma_saddlePoint} 
shows that there is a third point $q$ which is critical but not a minimum. This in turn contradicts
Lemma~\ref{lemma_positiveHessianTransversalToCriticalOrbitInComplexifiedGorbit}.
\end{proof}
\subsection{Existence of a local potential}
Before proving Theorem~\ref{theorem_localPotentialAroundM}, namely the existence of a local potential,
we will treat
first the simpler product case.
For this purpose we have to consider
a real, closed $(1,1)$-form $\omega$ which might not be a K\"ahler form
but admit a map $\mu:X\to\lie{G}^*$ satisfying $\de\mu^v=\imath_{\tilde v}\omega$ which we 
still call a momentum map. In this context, we call an
invariant plurisubharmonic function $\rho$
which might not be strictly plurisubharmonic
but induces
$\omega$ and $\mu$ still a potential.
\begin{lemma}
Let $\gtube$ be a $G$-tube associated to a
Lie group $G$.
Furthermore let $S$ be a 1-connected Stein manifold and let $G$ act on $X=\gtube\times S$
on the first factor. Let a $G$-invariant real closed $(1,1)$-form $\omega$ be given 
on $X$ with a momentum
map $\mu$. Then there is a potential $\rho$ on $X$. 
\label{lem_productCase}
\end{lemma}
\begin{proof}
We will treat the two factors $\gtube$ and $S$ separately, construct two functions $\rho_G$ and $\rho_S$ 
associated to the directions and finally add them to the desired function $\rho=\rho_G+\rho_S$.\\
For the factor $S$ define the embedding
\begin{align*}
 i_S:S&\to X\\
s&\mapsto (e,s)
\end{align*}
with $e$ denoting the neutral element in $G\subset \gtube$. Since $S$ is 1-connected, there is a 1-form $\alpha$
on $S$ such that $\imath_S^*\omega=\de\alpha$ and since $S$ is a
Stein manifold, it follows that there is a
a smooth
function $\rho_S$ on $S$ fulfilling
\begin{align*}
\de\dc\rho_S=i_S^*\omega
\end{align*}
We extend $\rho_S$ 
constantly to the $\gtube$-factor.\\
Now let us turn to the $\gtube$-direction.
We will show that there is a unique smooth function $\rho_G$ on
some $G$-stable neighborhood of $S\subset X$ satisfying for all $v\in\lie{G}$
\begin{align}
\tilde v(\rho_G)&= 0\label{eq_condInv}\\
J\tilde v(\rho_G)&=\mu^v\label{eq_condMoment}
\end{align}
and normalized on $S$ to
\begin{align}
\rho_G\vert_S&= 0\label{eq_condNormalization}
\end{align}
In order to establish this we define a 1-form $\eta$ by
\begin{align}
\eta(\tilde v)&=0\quad\text{for all}\ v\in\lie{G}\notag\\
\eta(J\tilde v)&=\mu^v\quad\text{for all}\  v\in\lie{G}\notag\\
\eta(\xi)&=0\quad\text{for all}\ \xi\in T_{(\gamma,s)}S\notag
\end{align}
where $T_{(\gamma, s)}S$ denotes the tangent vectors on $\gtube\times S$ tangential to the $S$-direction. 
We claim that pulled back to each $\gtube$-direction, i.e. pulled back via
$\varphi_s:\gtube\to \gtube\times S, \gamma\mapsto (\gamma, s)$,
the form is closed, i.e.
$\de\varphi_s^*\eta=0$ for all $s\in S$.
For this we use $[\lieder_\zeta,\imath_\xi]=0$. Applied to the 1-form $\eta$ gives
\begin{equation}
\de\eta(\zeta,\xi)=\zeta(\eta(\xi))
-\xi(\eta(\zeta))-\eta([\zeta,\xi])\label{equation_de_eta_zeta_xi_for1form}
\end{equation}
for arbitrary vector fields $\zeta,\xi$.
For $v,w\in\lie{G}$ we obtain
\begin{equation}
\de\eta(\tilde v,\tilde w)=0\ ,\label{equation_etaClosedvw}
\end{equation}
since $\eta(\tilde v)=\eta(\tilde w)=\eta([\tilde v,\tilde w])=0$.
For the next step 
we observe that $J\tilde v(\mu^w)=\omega(\tilde w, J\tilde v)$,
$\omega(\zeta,\xi)=\omega(J\zeta,J\xi)$ and $[J\tilde v,J\tilde w]=-[\tilde v,\tilde w]$. Thus
\begin{eqnarray}
\de\eta(J\tilde v,J\tilde w)&= &\omega(\tilde w, J\tilde v)-\omega(\tilde v,J\tilde w)-\eta([J\tilde v,J\tilde w])\notag\\
&=&\omega(\tilde w, J\tilde v)+\omega(J\tilde v,\tilde w)=0\label{equation_etaClosedJvJw}
\end{eqnarray}
And finally using $[\tilde v,J\tilde w]=J[\tilde v,\tilde w]$ and $[\tilde v,\tilde w]=-\widetilde{[v,w]}$
we obtain
\begin{eqnarray}
\de\eta(\tilde v,J\tilde w)&= &\omega(\tilde w, J\tilde v)-J\tilde w(\eta(\tilde v)-\eta(J[\tilde v,\tilde w])\notag\\
&=&\omega(\tilde v,\tilde w)-\mu^{-[v,w]}=0\label{equation_etaClosedvJw}
\end{eqnarray}
where the last equality is the infinitesimimal version of the $G$-equivariance of $\mu$. Since
$\tilde v$ and $J\tilde v$ combined for all
$v\in\lie{G}$ span each tangent space $T_\gamma\gtube$ the equations \eqref{equation_etaClosedvw},
\eqref{equation_etaClosedJvJw} and
\eqref{equation_etaClosedvJw} show that the
form $\varphi_s^*\eta$ is closed for all $s$.\\
We choose some 1-connected open set $\Omega_G\subset G$ containing $e$. The set 
\begin{equation*}
\Omega=\Omega_G\times\Sigma\subset G\times\Sigma\cong\gtube
\end{equation*}
is 1-connected as well and
for all $\gamma \in \Omega$
we can define
the smooth function $\rho_G$ by
\begin{equation*}
\rho_G(\gamma,s)=\int_e^\gamma\varphi_s^*\eta
\end{equation*}
which satisfies equations~\eqref{eq_condInv},\eqref{eq_condMoment} and 
\eqref{eq_condNormalization}. By construction $\rho_G$ is locally $G$-invariant and can therefore be extended
$G$-invariantly to $\gtube\times S$.\\
The desired function will be $\rho=\rho_G+\rho_S$.
The last step is to see that
$\de\dc\rho=\omega$ holds.
We will establish this by applying both sides to all possible vector fields. From local $G$-invariance of $\rho$, namely
\begin{align*}
0&=\lieder_{\tilde v}\dc\rho=\de\imath_{\tilde v}\dc\rho+\imath_{\tilde v}\de\dc\rho\\
\end{align*}
we deduce for all $v\in\lie{G}$
\begin{align*}
\imath_{\tilde v}\de\dc\rho&=- \de\imath_{\tilde v}\dc\rho=\de\mu^v=\imath_{\tilde v}\omega
\end{align*}
Thus for any vector field $\xi$
\begin{align*}
\imath_\xi\imath_{\tilde v}\de\dc\rho&=\imath_\xi\imath_{\tilde v}\omega
\end{align*}
Similarly we obtain
for any vector field $\xi$
\begin{align*}
\imath_\xi\imath_{J\tilde v}\de\dc\rho&=-\imath_ {J\xi}\imath_{\tilde v}\de\dc\rho\\
&=-\imath_{J\xi}\imath_{\tilde v}\omega=\imath_{\xi}\imath_{J\tilde v}\omega
\end{align*}
So we are left to the case of two vector fields $\zeta,\xi\in \Gamma(TS)$ seen as
vector fields on $X$ independent of the $\gtube$-direction. On $S\subset X$ 
\begin{align}
\imath_\zeta\imath_\xi\de\dc\rho&=\imath_\zeta\imath_\xi\omega\label{eq_omegaEqualsTau}
\end{align}
holds by construction. $G$-invariance of $\rho$ and $\omega$ extends 
this equation to $G\times S$. Finally, the equation holds on $X$ entirely, if for all $v\in\lie{G}$
\begin{align*}
\lieder_{J\tilde v}(\imath_\zeta\imath_\xi\de\dc\rho)&=\lieder_{J\tilde v}(\imath_\zeta\imath_\xi\omega)
\end{align*}
holds.
It is sufficient to establish equation~\eqref{eq_omegaEqualsTau} for commuting vector fields $\zeta$ and $\xi$, so
for simplicity we may assume $[\zeta,\xi]=0$ in the sequel. We will use twice
the following formula for 
general vector fields $X$ and $Y$ and a general 1-form $\eta$
\begin{equation*}
\imath_X\imath_Y\de\eta=\imath_X\de\imath_Y\eta-\imath_Y\de\imath_X\eta-\imath_{[X,Y]}\eta
\end{equation*}
as already mentioned above as equation~\eqref{equation_de_eta_zeta_xi_for1form}
as well as $\omega(JX,JY)=\omega(X,Y)$ for the real $(1,1)$-forms $\de\dc\rho$ and $\omega$
as well as the fact that $\zeta$ and $\xi$ commute with $J\tilde v$.
\begin{align*}
\lieder_{J\tilde v}(\imath_\zeta\imath_\xi\de\dc\rho)&=\imath_\zeta\imath_\xi(\lieder_{J\tilde v}\de\dc\rho)\\
&=\imath_\zeta\imath_\xi\de\imath_{J\tilde v}\de\dc\rho\\
&=\imath_\zeta\de\imath_\xi\imath_{J\tilde v}\de\dc\rho-\imath_\xi\de\imath_\zeta\imath_{J\tilde v}\de\dc\rho\\
&=-\imath_\zeta\de\imath_{J\xi}\imath_{\tilde v}\de\dc\rho+\imath_\xi\de\imath_{J\zeta}\imath_{\tilde v}\de\dc\rho\\
&=-\imath_\zeta\de\imath_{J\xi}\de\mu^v+\imath_\xi\de\imath_{J\zeta}\de\mu^v\\
&=-\imath_\zeta\de\imath_{J\xi}\imath_{\tilde v}\omega+\imath_\xi\de\imath_{J\zeta}\imath_{\tilde v}\omega\\
&=\imath_\zeta\de\imath_\xi\imath_{J\tilde v}\omega-\imath_\xi\de\imath_\zeta\imath_{J\tilde v}\omega\\
&=\imath_\zeta\imath_\xi\de\imath_{J\tilde v}\omega\\
&=\imath_\zeta\imath_\xi\lieder_{J\tilde v}\omega\\
&=\lieder_{J\tilde v}(\imath_\zeta\imath_\xi\omega)
\end{align*}
Thus $\rho$ is a $G$-invariant function fulfilling both conditions
\begin{eqnarray*}
\omega&=&\de\dc\rho\\
\mu^v&=&-\imath_{\tilde v}\dc\rho\ .
\end{eqnarray*}
\end{proof}
Now we consider the general situation.
We aim to prove the existence of a $G$-invariant
potential on some $G$-stable neighborhood of each point $x\in\mathcal{M}$.
\begin{proof}[Proof of Theorem~\ref{theorem_localPotentialAroundM}]	
Theorem~\ref{theorem_slice} shows that each point in $\mathcal{M}$ admits a neighborhood
isomorphic to a local slice model, therefore we are left
to the slice situation.
For simplicity we may assume
\begin{equation*}
X=\gtube\timescat{K} S
\end{equation*}
Hereby $S$ as well as $\gtube$ can be shrunk if necessary provided the latter set stays $G$-stable.
The point $x$ corresponds to 
$[e,s]\in\gtube\timescat{K} S$.
Define the projection
\begin{equation*}
\pi:\gtube\times S\to \gtube\timescat{K} S
\end{equation*}
and lift the objects to the product $\gtube\times S$, namely
$\hat\omega=\pi^*\omega$ and $\hat\mu=\pi^*\mu$
to some $G$-stable neighborhood
of $(e,s)$. Lemma~\ref{lem_productCase} provides a $G$-invariant
plurisubharmonic potential $\hat \rho$ on some $G$-stable
neighborhood of $(e,s)$ which can be chosen $K$-stable.\\
Our aim is to push down $\hat\rho$ to $X$. Therefore we have to show that $\hat\rho$ is $K$-invariant and locally $K^\C$-invariant with respect
to the diagonal $K$-action.\\
In a first step, 
$\hat\rho$ can be made $K$-invariant by averaging.
So for any induced vector field $v\in\lie{K}$ of the diagonal action, $\lieder_{\tilde v}\hat\rho=0$ and consequently
$\lieder_{\tilde v}\dc\hat\rho=0$. We conclude that
\begin{align}
 \imath_{\tilde v}\dc\hat\rho\ \text{is constant}\label{eq_vdcrhoConst}
\end{align}
since
\begin{align*}
\de\imath_{\tilde v}\dc\hat\rho=\lieder_{\tilde v}
\dc\hat\rho-\imath_{\tilde v}\de\dc\hat\rho=
-\imath_{\tilde v}\hat\omega=0
\end{align*}
By assumption $\hat\mu(e,s)=0$, so we get
\begin{align*}
 \imath_{\tilde v}\dc\hat\rho_{(e,s)}&=-\hat\mu(e,s)=0
\end{align*}
and with observation~\eqref{eq_vdcrhoConst} as desired
\begin{align*}
 \imath_{\tilde v}\dc\hat\rho=0
\end{align*}
This shows
\begin{align*}
J\tilde v(\hat\rho)&=0\quad\text{and}\\
\tilde v(\hat\rho)&=0
\end{align*}
the latter due to invariance. Thus $\hat\rho$ vanishes on the fibers of $\pi$ since
each fiber is a
single complexified $K$-orbit according to
Proposition~\ref{prop_shrinkingUnderFreeAction}\eqref{inProp_singleComplKorbit},
and therefore pushes down to the
quotient $X=\gtube\timescat{K} S $
fulfilling the desired properties.
\end{proof}
Theorem~\ref{theorem_localPotentialAwayMomentZero} states the existence of local potentials
away from $\mathcal{M}$, if the orbit 
$G\cdot x$ is totally real, the form $\omega$ pulled back to $G\cdot x$ is exact and
the first
de Rham cohomology group $H^1_{dR}(G^0)$ of the identity connected component of $G$ is zero.
\begin{proof}[Proof of Theorem~\ref{theorem_localPotentialAwayMomentZero}]
The $G$-orbit $G\cdot x$ being totally real, we can apply again
the local slice theorem, lift $\omega$ and $\mu$ to the product
$\gtube\times S$ and construct a plurisubharmonic function $\hat\rho$ by Lemma~\ref{lem_productCase} as above, again $K$-invariant
with respect to the diagonal action. The same argument provides that
\begin{align*}
 \imath_{\tilde v}\dc\hat\rho\ \text{is constant.}
\end{align*}
The exactness of $\omega$ on the considered $G$-orbit implies exactness on some neighborhood, so after further shrinking of $X$ we have
$\omega=\de\beta$. For $\hat\beta=\pi^*\beta$ and
any induced vector field $v\in\lie{K}$ of the diagonal action we have
\begin{align*}
\imath_{\tilde v}\hat\beta=0
\end{align*}
Thus the form $\eta=\dc\hat\rho-\hat\beta$ is closed
and therefore exact since $H^1_{dR}(X_0)=H^1_{dR}(G^\circ)=0$
for each connected component
$X_0$ of $\gtube\times S$.
This shows that $\eta=\de f$ for some $K$-invariant
function $f$, this implies
\begin{align*}
\imath_{\tilde v}\dc\hat\rho=
\imath_{\tilde v}\dc\hat\rho-\imath_{\tilde v}\hat\beta
=\imath_{\tilde v}\eta=\tilde v(f)=0
\end{align*}
and we obtain
\begin{align*}
J\tilde v(\hat\rho)&=0\quad\text{and}\\
\tilde v(\hat\rho)&=0\ .
\end{align*}
Therefore the potential $\hat\rho$ pushes down to the 
quotient \mbox{$X=\gtube\timescat{K} S$}
fulfilling the desired properties. 
\end{proof}
\subsection{Modifying potentials}
We start by showing that
a $G$-ex\-haus\-tive strictly plurisubharmonic
function can be modified
to be greater than any given $G$-invariant
continuous function.
\begin{lemma}
Let $\nu$ be a $G$-ex\-haus\-tive strictly
plurisubharmonic function on $X$ with connected image
$I=\nu(X)\subset \R$. Given a continuous
$G$-invariant function $\lambda:X\to \R$ , then there is a function
$\chi:I\to\R$ such that 
$\nu_\chi=\chi\circ\nu$ is strictly plurisubharmonic,
$G$-exhaustive with
$\nu_\chi>\lambda$ and $\sup_X\nu_\chi=\infty$.\\
If furthermore there is a $G$-stable open subset $U\subset X$ with $\lambda(x)\leq\nu(x)$ for
all $x\in U$, the function
$\nu_\chi$ can be chosen such that
$\nu_\chi(x)=\nu(x)$ for all $x\in U$.
\label{lemma_bendingUpPotential}
\end{lemma}
\begin{proof}
We calculate
\begin{equation}
\begin{array}{lclll}
\de\dc\nu_\chi&=&\de(\chi'\circ\nu\cdot\dc\nu)\\
&=&\chi''\circ\nu\cdot\de\nu\wedge\dc\nu&+&\chi'\circ\nu\cdot\de\dc\nu
\end{array}
\end{equation}
Since $\de\dc\nu$ is a K\"ahler form and
$\de\nu\wedge\dc\nu$ is a non-negative $(1,1)$-form
$\nu_\chi$ is strictly plurisubharmonic whenever
$\chi'$ and $\chi''$ are positive.
We define the function
\begin{equation}
\begin{array}{rccl}
\alpha: &I&\to & \R\\
&t&\mapsto&\sup\{\lambda(x)\ \vert\ \nu(x)=t\}
\end{array}
\end{equation}
which is continuous. Demanding $\chi>\alpha$ implies
$\nu\chi>\lambda$. But there is in fact a smooth function
$\chi$ on $I$ with $\chi>\alpha$, $\chi'>0$, $\chi''>0$ and $\sup\chi=\infty$.
Furthermore for all $t\in\nu(U)$ we may
choose $\chi(t)=t$, since $\alpha(t)\leq t$.
\end{proof}
This next lemma will be used to construct
$G$-exhaustive strictly plurisubharmonic functions
on a local slice model which coincides on a 
sufficiently large set with a given strictly plurisubharmonic function.
\begin{lemma}
Let $Y=\gtube\timescat{K}S$ be a local slice model. Then
\begin{enumerate}
\item there is a $G$-exhaustive strictly
plurisubharmonic function $\nu:Y\to\R$ which can be chosen to be greater than a given $G$-invariant continuous function $\lambda$ and\label{lemma_part_exhaustionOnLocalSliceModel}
\item given a $G$-invariant strictly plurisubharmonic function
$\rho:Y\to\R$, a $G$-stable
open subset $U\subset Y$ with
$U/G\subset\subset Y/G$ 
there is a $G$-exhaustive strictly
plurisubharmonic function $\nu_{U,\rho}:Y\to\R$ such
that $\nu_{U,\rho}\vert_U = \rho\vert_U$.\label{lemma_part_potentialModifiableToExhaustionOnLocalSliceModel}
\end{enumerate}\label{lemma_existenceOfExhaustivePotentialOnSliceModel}
\end{lemma}
\begin{proof}
By construction there is a $G$-exhaustive
strictly plurisubharmonic function $\nu_G$ on $\gtube$
and an
exhaustive strictly plurisubharmonic function $\nu_S$ on $S$. Lemma~\ref{lemma_bendingUpPotential}
shows that we may assume after modification the supremum of both functions to be infinity.
We may assume $\nu_S$ to be $K$-invariant and $\nu_G$ invariant with respect to the $K$-action from the right, both can be obtained by averaging. On the product, we may define
\begin{equation}
\begin{array}{rccl}
\nu_P:&\gtube\times S&\to&\R\\
&(\gamma,s)&\mapsto&\nu_G(\gamma)+\nu_S(s)
\end{array}
\end{equation}
which is $K$-invariant with respect to the diagonal
$K$-action. It induces a K\"ahler form $\omega_P$ and a
momentum map $\mu_P$ with respect to the $K$-action.
As the $K$-action is free the zero level
$\mathcal{M}_P$
is a (real) submanifold which is both $K$- and $G$-stable
as $\nu_P$ is $(G\times K)$-invariant. By Corollary~\ref{corollary_GexhaustionPushesDowntoMmodKforFreeKaction}
the restriction $\nu_P\vert_{\mathcal{M}_P}$
pushes down
to $\mathcal{M}_P/K$ 
and by the identification
$\mathcal{M}_P/K\cong \gtube\timescat{K}S$
to a strictly plurisubharmonic $G$-exhaustion
$\nu:\gtube\timescat{K}S\to\R$.
This proves the first assertion in part~\eqref{lemma_part_exhaustionOnLocalSliceModel} while the second assertion follows from Lemma~\ref{lemma_bendingUpPotential}.\\
Given a $G$-invariant strictly plurisubharmonic
function $\rho$, a $G$-stable open
subset $U\subset Y$
with $U/G\subset\subset Y/G$
and a $G$-exhaustive strictly plurisubharmonic function $\nu$ established from part~\eqref{lemma_part_exhaustionOnLocalSliceModel}.
Choose a smooth $G$-invariant function 
$\chi:Y\to [0,1]$ such that with
$V=\{y\in Y\ \vert\ \chi(y)\not =0\}$ the set
$V/G$ is relatively compact in $Y/G$ and $\chi\vert_U= 1$.
Since $\rho$ is strictly plurisubharmonic,
there is a $t_0>0$ such that
$\rho-t_0\cdot\chi\cdot\nu$ is strictly plurisubharmonic.
Lemma~\ref{lemma_bendingUpPotential}
provides a $G$-exhaustive strictly
plurisubharmonic function $\nu_0$ such that
$\nu_0\vert_U=\nu\vert_U$ and
$\nu_0>\nu+\left\vert\frac{\rho}{t_0}\right\vert$ on
$Y\backslash \bar V$. The function
$\nu_{U,\rho}=\rho-t_0\cdot\chi\cdot\nu+t_0\cdot\nu_0$ 
equals $\rho$ on $U$ and is greater than the $G$-exhaustion $\nu$ outside $\bar V$ and therefore a $G$-exhaustion.
As a sum of two strictly plurisubharmonic
functions it is strictly plurisubharmonic.
This proves part~\eqref{lemma_part_potentialModifiableToExhaustionOnLocalSliceModel}.
\end{proof}

%%%%%%%%%%%%%%%%%%%%%%%%%%%%%%%%%%%%%%%%%%%%%
%
%   5   Complex structure on symplectic reduction
%
%%%%%%%%%%%%%%%%%%%%%%%%%%%%%%%%%%%%%%%%%%%%%

\section{Complex structure on symplectic reduction}\label{section_cplxStructure}
The purpose of this section is to give the proofs
for the main Theorem~\ref{theorem_quotientLocallyGivenFromHilbertQuotient}
and therewith Theorem~\ref{theorem_quotientComplexStructure}. Recall that
Theorem~\ref{theorem_quotientLocallyGivenFromHilbertQuotient} states that
for every point $x\in\mathcal{M}$ there is a $G$-stable neighborhood
$Y\subset X$
such that $Y\modmod G$ is a (reduced) normal complex space and a $G$-stable open neighborhood
$U\subset Y$
with $U/G\subset\subset Y/G$ and
$\pi(U)=U\modmod G$
defining $\mathcal{M}_U=\mathcal{M}\cap U$
the diagram
\begin{eqnarray}
\mathcal{M}_U&\stackrel{i_U}{\hookrightarrow}& U\subset Y\notag\\
\downarrow p&\quad&\downarrow\pi\label{eqn_recall_localRealizationInTheorem}\\
\mathcal{M}_U/G&\stackrel{\varphi_U}{\longrightarrow}& U\modmod G
\subset Y\modmod G\notag
\end{eqnarray}
commutes such that $\varphi_U:\mathcal{M}_U/G\to
U\modmod G$
is a homeomorphism
and biholomorphic in the sense that it induces
an isomorphism $(\varphi_U)^*$ between the
sheaves $\mathcal{O}_{Y\modmod G}$ and $\mathcal{O}_{\mathcal{M}/G}$ whenever the map is defined.\\
In the following Section~\ref{subsection_localHomeomorphism}
we establish first that the induced map $\varphi_U$ in diagram~\eqref{eqn_recall_localRealizationInTheorem}
is a homeomorphism. In section~\ref{subsection_localBiholomorphism}
we will show the main result that $\varphi_U$ defines a biholomorphism, i.e. 
local correspondences
providing the complex structure.
\subsection{Local homeomorphism to local analytic Hilbert quotient}
A key ingredient for the proof of the topological part of
Theorem~\ref{theorem_quotientLocallyGivenFromHilbertQuotient} is the existence of a $G$-invariant
strictly plurisubharmonic potential on a local slice model and its behaviour on closed and on non-closed
complexified $G$-orbits. 
The existence of such a function has already been
established, so in the sequel we will analyse
its behaviour.
\begin{theorem}
Let $Y=\gtube\timescat{K}S$ be a local slice model and $\rho$ a $G$-exhaustive
strictly plurisubharmonic function
inducing a K\"ahler form $\omega$, a momentum map $\mu$
and the zero level $\mathcal{M}=\mu^{-1}(0)$. Then
for each fiber $F$ of $\pi:Y\to Y\modmod G$
the function $\rho$ becomes critical where it attains a
local minimum. This is exactly
on a single $G$-orbit contained in the unique closed complexified $G$-orbit in $F$ and this $G$-orbit
is the set $F\cap\mathcal{M}$. The map
$\varphi:\mathcal{M}/G\to Y\modmod G$ induced from the inclusion $\mathcal{M}\hookrightarrow Y$ is a
homeomorphism.
\label{theorem_homeomorphismInLocalSliceModel}
\end{theorem}
The first step will be to 
analyse the following special case of a local slice model. 
So in the sequel, the special local slice model
$Z=\gtube\timescat{L}\Delta$ is considered where
$L\subset G$ is isomorphic to 
$S^1$ and $\Delta$ is a bounded
connected open subset of the origin in $\C$ on which $L$ acts non-trivially as the restriction of a complex
linear representation.
Let $z_0=[e,0]\in Z$ then $B (z_0)=
\gtube\timescat{L}\{0\}$ is a closed complexified
$G$-orbit in $Z$ and let
$z_1\not\in B (z_0)$ then
$B (z_1)=\gtube\timescat{L}(\Delta\backslash\{0\})$ is the complement and henceforth an open
complexified $G$-orbit. Thus
this model is special as
on the one hand side $Z$ is a manifold and 
on the other hand side 
$Z$ splits into just two complexified $G$-orbits, the closed submanifold $B (z_0)$ and 
its open complement $B(z_1)$.
We will show in the next two lemmata that
a $G$-exhaustive strictly plurisubharmonic function has a unique
critical $G$-orbit in $B(z_0)$ and no critical point in the open set $B(z_1)$.
\begin{lemma}
Let $\rho$ be a $G$-exhaustive
strictly plurisubharmonic function on the local
slice model $Z=\gtube\timescat{L}\Delta$ introduced above.
Then there is a unique critical
orbit $G\cdot p\subset {B(z_0)}$ for
$\rho\vert_{B(z_0)}$
and for each local submanifold $\Sigma_Z$ through $p$
transversal to the $G$-orbit $G\cdot p$ the Hessian of $\rho\vert_{\Sigma_Z}$
is strictly positive and the orbit is a local isolated minimum locus of $\rho$ in $Z$.\label{lemma_isolatedAndUniqueMinimumLocusInClosedCplxfdOrbitInSpecialLocalSliceModel}
\end{lemma}
\begin{proof}
The subset ${B(z_0)}$ is a closed
submanifold which in turn is
a complexified $G$-orbit as it is $G$-equivariantly biholomorphic to $\gtube\modmod L$. The function $\rho$ is assumed to be a $G$-exhaustion, so is the function $\rho\vert_{{B(z_0)}}$, and hence there is a point
$p\in {B(z_0)}$ with respect to which $\rho\vert_{{B(z_0)}}$ is critical.
By Lemma~\ref{lemma_positiveHessianTransversalToCriticalOrbitInComplexifiedGorbit}, the $G$-orbit $G\cdot p$ is a local isolated
minimum locus of $\rho\vert_{{B(z_0)}}$.
Corollary~\ref{corollary_uniqueCriticalOrbitInComplexifiedGorbitIfExhaustive}
shows that this orbit $G\cdot p$ is the only critical orbit of $\rho\vert_{{B(z_0)}}$
providing the uniqueness assertion claimed in the lemma.
This local minimum locus is
of maximal dimension, i.e.
\begin{equation*}
\dim_\R G\cdot p=\dim_\C\gtube\modmod L=\dim_\C\gtube-1=\dim_\R G-1
\end{equation*}
So this action has a
1-dimensional isotropy group at $p$. Denote $H$ its identity component. The group $H$ is compact and therefore isomorphic to $S^1$.\\
We claim
that $p$ is a critical point of $\rho:Z\to\R$.
In order to see this we first note that $G\cdot p$
is $H$-stable. 
Thus the subspace $T=T_p(G\cdot p)$ is $H$-stable as
well as $V=JT$. By construction
$T\oplus V=T_pB(z_0)$.
Further we may choose an $H$-stable complement
$W\cong \C$ to $T\oplus V$ in $T_pZ$.
Keep in mind that the $H$-representation on $W$
is non trivial.
On the one hand side $p$ is selected to be
a critical point of $\rho\vert_{B(z_0)}$, so
$\de\rho$ vanishes on $T\oplus V$. On the other hand $\de\rho$ is $H$-invariant on $W\cong\C$ hence vanishes on $W$ as well. This shows the claim.\\
In order to simplify the argumentation we will linearize
the setting by pulling it up to a neighborhood
of the origin in $T_pZ$ in the next step.
There is an $H$-equivariant biholomorphism
from an open neighborhood of the origin
in $T_pZ$ to an open neighborhood
of $p$ in $Z$ mapping the origin to $p$
such that the derivative
$T_0T_pZ\to T_pZ$ becomes the identity
after the 
natural identification $T_0T_pZ\cong T_pZ$.
In the sequel we will argue in $T_pZ$ and denote $\tilde\rho$ the pullback of $\rho$ to an open neighborhood of the origin there.\\
Our purpose is now to show that $\tilde\rho$ restricted to $V\oplus W$ has strictly positive Hessian and the $G$-orbit is therefore a local isolated minimum locus of $\rho$. For this, we
will analyse the second order terms of $\tilde\rho$ at $0$, namely we set
\begin{equation*}
 \Theta(v_1,v_2)=v_1(v_2(\tilde\rho))_{0}\qquad\text{for}\ v_1,v_2\in T_0T_pZ\cong T_{p}Z
\end{equation*}
and have to show that the symmetric bilinear
form $\Theta$
has $T$ as its kernel and that the restriction
$\Theta\vert_{V\oplus W}$ is positive.\\
We denote by $\Sym^2(E)$ the space of
(real) symmetric bilinear forms on a vector
space $E$ and
by $(E_1^*\symtensor E_2^*)$ the space generated by 
elements of the form $e_1^*\symtensor e_2^*=
\frac{1}{2}(e_1^*\otimes e_2^*+e_2^*\otimes e_1^*)\in\Sym^2(E_1\oplus E_2)$ where $e_j\in E_j^*$. Note that the decomposition
$E=\bigoplus_{j=1}^N E_j$ induces a decomposition
\begin{equation}
 \Sym^2(E)=\bigoplus_{j=1}^N\Sym^2(E_j)\oplus\bigoplus_{{\substack{j,k=1\\ j<k}}}^N\left(E_j^*\symtensor E_k^*\right)\label{equation_decompositionSymmetricBilinearForm}
\end{equation}
If additionally $E$ is an $H$-representation, we denote the set of invariant elements in the respective sets by the exponent $H$, i.e. $\Sym^2(E)^H$. If $E=\bigoplus E_j$
is an $H$-invariant decomposition, then
the space of invariant symmetric bilinear forms
is the sum of invariant elements with respect to the decomposition~\eqref{equation_decompositionSymmetricBilinearForm}.\\
We note that $\Theta$ is $H$-invariant. The aim in the sequel will be to find a decomposition $V=V_W\oplus V_N$ into $H$-stable subspaces and consequently 
$T_W=J\cdot V_W, T_N=J\cdot V_N$
such that 
with respect to the decomposition
\begin{equation*}
T_{p}Z=T_W\oplus T_N\oplus V_W\oplus V_N\oplus W
\end{equation*}
the form $\Theta$ decomposes as
$\Theta=\Theta_W+\Theta_N$ with
\begin{equation*}
\Theta_W\in\Sym^2(T_W\oplus V_W\oplus W)^H
\quad\text{and}\quad
\Theta_N\in
\Sym ^2(T_N\oplus V_N)^H
\end{equation*}
and further all the components
$\Sym^2(T)^H$, $(T^*\symtensor V^*)^H$,
$(T^*\symtensor W^*)^H$ 
of $\Theta$ and hence $\Theta_W$ vanish while
for each of the remaining spaces
\begin{equation*}
\Sym^2(V_W)^H, \Sym^2(W)^H\ \text{and}\ (V_W^*\symtensor W^*)^H
\end{equation*}
of the $H$-invariant elements
we will be able to find a particular form of the
corresponding components of $\Theta_W$ 
such that
finally we are reduced to an ansatz for $\Theta_W$
with $3$ parameters
for which we establish the positivity of the
Hessian.\\
Let us start to see that $\Theta$ does vanish
on $T$, i.e. that all components of $\Theta$ in
$\Sym^2(T)^H$, $(T^*\symtensor V^*)^H$ and
$(T^*\symtensor W^*)^H$ vanish. For a given vector
$v_2\in T$ there is a fundmental vector field $\tilde v$
such that $\tilde v_p=v_2\in T$. But as local $G$-invariance reads $\tilde v(\rho)= 0$ the linear map
$\Theta(\,.\,,v_2)$ vanishes.\\ 
Now, define the following $H$-stable linear subspace.
\begin{equation*}
V_N=\{v\in V\ \vert \ \Theta(v,w)=0
\ \text{for all}\ w\in W\}
\end{equation*}
In the easy case where $V_N=V$ we may write $\Theta
=\Theta_N+\Theta_W$ with $\Theta_N\in\Sym^2(V)^H$ and
$\Theta_W\in\Sym^2(W)^H$ where the first has positive
Hessian according to Lemma~\ref{lemma_positiveHessianTransversalToCriticalOrbitInComplexifiedGorbit} and $\Theta_W$ is, due to the $H$-invariance, just a positive multiple
of the squared modulus on $W\cong\C$
which has positive Hessian.\\
Now we turn to the case where $V_N$ is a proper subspace of $V$. Since $W$ is of real dimension~2, the space $V_N$ has at most codimension~2
in $V$.
Restricted to $V$ the 
bilinear form $\Theta\vert_V$ is strictly positive
(Lemma~\ref{lemma_positiveHessianTransversalToCriticalOrbitInComplexifiedGorbit}).
Let $V_W$ be the orthogonal
complement of $V_N$ with respect to $\Theta\vert_V$
so that $V=V_W\oplus V_N$.
Thus there are $H$-invariant bilinear forms
$\Theta_W\in\Sym^2(V_W\oplus W)^H$ and
$\Theta_N\in\Sym^2(V_N)^H$ 
such that $\Theta=\Theta_W+\Theta_N$.\\
In order to analyse $\Theta_W$ observe that there 
is an element $\varepsilon\in H$ whose induced endomorphism on $W$ coincides with
multiplication by $i$.
We choose a non-vanishing linear map $w^*\in W^*$.
Since $w^*, \varepsilon^*w^*$ form a (real) basis of $W^*$, there
are functionals $v^*,\tilde v^*\in V^*$
such that the component $\Theta_M$ of $\Theta_W$ in
$(V_W^*\symtensor W^*)$ equals
$\Theta_M=v^*\symtensor w^*+ \tilde v^*\symtensor \varepsilon^*w^*$.
We think of this expression as the ``mixed term''
of $\Theta$.
The element $\varepsilon$ was chosen such that
$\varepsilon^*\varepsilon^*w^*=-w^*$
and the equation $\varepsilon^*\Theta_M=\Theta_M$
holds due
to $H$-invariance of $\Theta_M$. We calculate
\begin{equation*}
\begin{array}{rrcccc}
 &\varepsilon^*\Theta_M&=&
 \varepsilon^*\tilde v^*\symtensor\varepsilon^*\varepsilon^*w^*&+&
 \varepsilon^*v^*\symtensor\varepsilon^*w^*\\
 &&=&
 -\varepsilon^*\tilde v^*\symtensor w^*&+&
 \varepsilon^*v^*\symtensor\varepsilon^*w^*\\
 =&\Theta_M&=&v^*\symtensor w^*&+& \tilde v^*\symtensor \varepsilon^*w^*
\end{array}
\end{equation*}
and conclude $\tilde v^*=\varepsilon^*v^*$.
This shows
\begin{equation}
\Theta_M=v^*\symtensor w^*+ \varepsilon^* v^*\symtensor \varepsilon^*w^*\ .
\label{equation_mixedTerm}
\end{equation}
Again using the invariance
$\varepsilon^*\Theta_M=\Theta_M$ assuming linear dependence over $\R$ and $\Theta_M\not=0$,
i.e. $\varepsilon^*v^*=\lambda\cdot v^*$, leads to the
contradiction $\lambda^2=-1$, thus
$V_W$ is 2-dimensional.
In summary, we developped a splitting of
$T\oplus V\oplus W$ into subspaces
such that the only term of $\Theta$ 
involving both $V$ and $W$
is a multiple of $\Theta_M$, a bilinear form on $V_W\oplus W$.\\
We now go back to the function $\tilde\rho$
having induced the bilinear form $\Theta$.
The second order terms of $\tilde\rho$ with respect to the linear structure on $T_pZ$ around the origin define a function $\theta:T_pZ\to\R$
which is $H$-invariant strictly plurisubharmonic
and induces the same bilinear form. With $T_W=JV_W$
the subspace $T_W\oplus V_W\oplus W$ is a
3-dimensional complex vector space,
$T_W\oplus V_W\cong\C^2$ and $W\cong\C$, and we
define the $H$-invariant strictly plurisubharmonic
function
$\theta_W=\theta\vert_{T_W\oplus V_W\oplus W}$.
We introduce a complex linear coordinate
$w=s+i\cdot t$ and (real) linear coordinates $x_1,x_2$
on $T_W$ inducing complex linear coordinates
$z_j=x_j+i\cdot y_j$ on $T_W\oplus V_W$.
With this the function $\theta_W$ splits into a sum of three functions, the first depending only on
the variables $x_1,x_2,y_1,y_2$, the third
depending only on $s,t$ and the second $\theta_M$ depending on the remaining ``mixed terms''.
In view of equation~\eqref{equation_mixedTerm}
the coordinates can be chosen such that there
is a $\alpha_2\in\R$ such that
\begin{equation*}
\theta_M=4\alpha_2\cdot\left(y_1s+y_2t\right)\ .
\end{equation*}
Further we established that $\theta_W$ does not depend on $x_1$ and $x_2$. So $H$-invariance forces this term to be a multiple of $y_1^2+y_2^2$. 
Finally the last term depending only on $s$ and $t$ 
has to be a multiple of $s^2+t^2$ again due to
$H$-invariance. In summary, there are
$\alpha_1,\alpha_2,\alpha_3\in\R$ such that
\begin{equation*}
 \theta_W=2\alpha_1(y_1^2+y_2^2)+4\alpha_2(y_1s+y_2t)
 +\alpha_3(s^2+t^2)
\end{equation*}
As $\theta_W$ is strictly plurisubharmonic on $T_W\oplus V_W\oplus W$ so are the restriction to both $T_W\oplus V_W$ and $W$, so $\alpha_1$ and $\alpha_3$ are positive. We calculate the Levi matrix with respect to the
coordinates $z_1,z_2,w$ as
\begin{equation*}
 \Levi\theta_W=\begin{pmatrix}
           \alpha_1&0&-i\cdot\alpha_2\\
           0&\alpha_1&\alpha_2\\
           i\cdot\alpha_2&\alpha_2&\alpha_3
          \end{pmatrix}
\end{equation*}
with determinant $\alpha_1(\alpha_1\alpha_3-2\alpha_2^2)$ which has to be positive by assumption that $\rho$ and
hence $\theta_W$ are strictly plurisubharmonic. In summary, we know
\begin{equation}
 \alpha_1>0,\quad\alpha_3>0\quad\text{and}\quad
 \alpha_1\alpha_3-2\alpha_2^2>0\label{equation_matrixPositivityConditions}
\end{equation}
Now, we calculate the Hessian retricted to the coordinates
$y_1,s,y_2,t$
\begin{equation*}
\Hess\theta_W=\begin{pmatrix}
                             4\alpha_1&4\alpha_2&0&0\\
                             4\alpha_2&2\alpha_3&0&0\\
                             0&0&4\alpha_1&4\alpha_2\\
                             0&0&4\alpha_2&2\alpha_3\\                             
                            \end{pmatrix}
\end{equation*}
This matrix is strictly positive if the positivity conditions~\eqref{equation_matrixPositivityConditions}
derived from the Levi matrix are fulfilled.\\
The Hessian of the other part $\theta_N=\theta\vert_{V_N}$ is strictly positive by Lemma~\ref{lemma_positiveHessianTransversalToCriticalOrbitInComplexifiedGorbit}, since the $G$-orbit is a local isolated minimum locus in ${B(z_0)}$. 
Finally, the constructed decomposition ensures
$\theta=\theta_W+\theta_N$, so
restricted to a local submanifold
transversal to the $G$-orbit the function $\theta$ has in fact a strictly positive Hessian.
Note that this holds for the easy case $V=V_N$
as well since in this case $\Theta_W=\alpha_3(s^2+t^2)$ with $\alpha_3>0$.
Therefore the $G$-orbit through $p$ is a local isolated minimum locus of $\rho$. 
\end{proof}
Now, we can use the particular structure of this special local slice model $Z=\gtube\timescat{L}\Delta$, namely that it is a manifold and
consists of just two complexified $G$-orbits, a closed one ${B(z_0)}$ and an open
one ${B(z_1)}$. 
\begin{lemma}
Let $G\cdot p_0$ be the critical orbit
in the closed complexified $G$-orbit $B(z_0)$ in the
special local slice
model $Z=\gtube\timescat{L}\Delta$
established in the previous Lemma~\ref{lemma_isolatedAndUniqueMinimumLocusInClosedCplxfdOrbitInSpecialLocalSliceModel}. This orbit
is the only critical and minimal locus
and hence a global minimum locus.
\label{lemma_uniqueCriticalOrbitInSpecialLocalSliceModel}
\end{lemma}
\begin{proof}
Let us assume there is a point $p_1\not\in
G\cdot p_0$
at which $\rho$ becomes critical.
As $G\cdot p_0$ is the only critical locus of
$\rho\vert_{B(z_0)}$ according to Corollary~\ref{corollary_uniqueCriticalOrbitInComplexifiedGorbitIfExhaustive} we know
$p_1\in B(z_1)$.
Lemma~\ref{lemma_saddlePoint}
shows that there is a further critical point
$q$ which is not a local minimum.
Again $q\not\in B(z_0)$ by Corollary~\ref{corollary_uniqueCriticalOrbitInComplexifiedGorbitIfExhaustive}. So the point $q$ lies in the open
complexified $G$-orbit $B(z_1)$. Lemma~\ref{lemma_positiveHessianTransversalToCriticalOrbitInComplexifiedGorbit} shows that 
restricted to a complexified
$G$-orbit $\rho$ becomes critical only if it becomes locally minimal. This contradicts the existence of $q$.
\end{proof}
After analysing the behaviour of
the function $\rho$ on a special local
slice model
the following lemma provides a first decisive link
of this special local slice model to the general
local slice
model.
\begin{lemma}
Let $\rho:\gtube\timescat{K}S\to\R$
be a $G$-exhaustive strictly plurisubharmonic smooth
function and $Z=\gtube\timescat{L}\Delta$ as in Lemma~\ref{lemma_isolatedAndUniqueMinimumLocusInClosedCplxfdOrbitInSpecialLocalSliceModel}
with $B(z_1)\subset Z$ the unique open complexified $G$-orbit.
If $\varphi:\gtube\timescat{L}\Delta\to
\gtube\timescat{K}S$ is a $G$-equivariant holomorphic map and if 
$q\in\varphi(Z)$ is a local minimum
for $\rho\vert_{B(q)}$, then
$q\not\in\varphi(B(z_1))$.
\label{lemma_liftingGeneralSliceModelToSpecialLocalSliceModel}
\end{lemma}
The proof is postponed. It consists in pulling
back $\rho$ by $\varphi$
and modifying this pulled back function with another
$G$-exhaustive strictly plurisubharmonic function $\eta$ on $Z$.
Assuming that $\rho$ admits a local minimum on a complexified $G$-orbit through a point
in $\varphi({B(z_1)})$ the
modified function on $Z$ will have a local isolated minimum
on ${B(z_1)}$. For this we need the
following two lemmata.\\
Recall that a proper $G$-action on a manifold $Z$
admits to a given point $z\in Z$
a (local) slice $\Sigma_z$ which can be characterized as follows. First the slice
is a local
submanifold of $Z$ containing $z$ 
and which is $G_z$-stable where $G_z$ denotes the isotropy group of
the $G$-action at $z$. Second the tangent space
$T_zZ$ splits
into $T_z(G\cdot z)$ and $T_z\Sigma_z$.
When chosen $\Sigma_z$ sufficiently small, this induces
a $G$-equivariant open embedding
\begin{equation*}
\begin{array}{ccc}
G\times^{G_z}\Sigma_z&\to&Z\\
{}[g,\sigma]&\mapsto&g\cdot\sigma\\
\end{array}
\end{equation*}
Here $G\times^{G_z}\Sigma_z$ is the quotient of $G\times \Sigma_z$
with respect to the following (diagonal) $G_z$-action.
\begin{equation*}
\begin{array}{rcccl}
\alpha:&G_z \times G\times\Sigma_z&\to&G\times\Sigma_z\\
&(h,(g,\sigma))&\mapsto&
(g\cdot h^{-1},h\cdot\sigma)
\end{array}
\end{equation*}
Finally, we recall that $\Sigma_z$ can be chosen
$G_z$-equivariantly diffeomorphic to a
$G_z$-stable open neighborhood of the origin
of $T_z\Sigma_z$ with $z$ mapped to the origin.
\begin{lemma}
Let $M$ and $N$ be manifolds, $G$ act properly
on $N$ and let
$\psi:M\to N$ be a $G$-equivariant submersion 
with $\psi(x)=y$. Then there are
respective slices $\Sigma_x$ and $\Sigma_y$
such that the restriction of $\psi$ becomes a
$G_x$-equivariant submersion
$\theta:\Sigma_x\to\Sigma_y$.
\label{lemma_slicesRelativeAlongSubmersion}
\end{lemma}
\begin{proof}
Observe that $G$ acts properly on $M$ as well.
We have to show that there are linear
subspaces $V\subset T_xM$ and $W\subset T_yN$
such that $V$ is complementray to $T_x(G\cdot x)$
and $W$ to $T_y(G\cdot y)$
and the restriction of $\psi_*:T_xM\to T_yN$
to $V$ becomes a $G_x$-equivariant submersion
$\theta:V\to W$. In order to construct this, we 
observe that $\psi(G\cdot x)=G\cdot y$. Hence the $G_x$-equivariant map
$\psi_*$ restricts to a $G_x$-equivariant
submersion $\psi_*\vert_{T_x(G\cdot x)}:T_x(G\cdot x)\to T_y(G\cdot y)$. 
Denote by $E$ the kernel
of $\psi_*\vert_{T_x(G\cdot x)}$ which is $G_x$-stable.\\
Now choose
a $G_y$-stable linear
subspace $W\subset T_yN$ complementary to
$T_y(G\cdot y)$. The $G_x$-stable preimage
$\left(\psi_*\right)^{-1}(W)$ contains $E$
as a $G_x$-stable linear subspace. Denote by $V$ a
$G_x$-stable complement to $E$ in 
$\left(\psi_*\right)^{-1}(W)$. The induced
restricted map $\theta:V\to W$ is a $G_x$-equivariant
submersion. Further $V$ is complementary to $T_x(G\cdot x)$ what finishes the proof.
\end{proof}
The next lemma is needed for
the proof of 
Lemma~\ref{lemma_liftingGeneralSliceModelToSpecialLocalSliceModel} deals with modifying a function
``in fiber direction'' when only having an isolated
minimum transversal to the fibers.
\begin{lemma}
Let the Lie group $G$ act properly on the manifolds
$M$ and $N$ and $\varphi:M\to N$ be a smooth, 
$G$-equivariant submersion with $\varphi(p)=q$.
Furthermore let $\eta:M\to\R$ and $\rho:N\to\R$ be
$G$-invariant smooth functions
such that $G\cdot q$ is a local isolated minimum locus of $\rho$ and $G\cdot p$ is a local isolated minimum locus of the restricted function
$\eta\vert_{\varphi^{-1}(G\cdot q)}$. Then for each
$G$-stable neighborhood $U$ of $p$ there is a
$G$-stable open neighborhood $V\subset U$ and a
$\delta>0$
such that for every $t\in(0,\delta)$ the function $\rho_t=\varphi^*\rho+t\cdot\eta$ attains a local isolated minimum on a $G$-orbit
in $V$.
\label{lemma_producingLocalMinimumOnLiftedFunction} 
\end{lemma}
\begin{proof}
By the
previous Lemma~\ref{lemma_slicesRelativeAlongSubmersion}
there are
slices
$\Sigma_p$ and $\Sigma_q$ for the $G$-action
at $p$ and $q$ respectively and a
$G_p$-equivariant submersion
$\theta:\Sigma_p\to\Sigma_q$.
In this setting 
$\rho_0=\rho\vert_{\Sigma_q}$ has a local isolated
minimum at $q$.
Let 
$\eta_0=\eta\vert_{\Sigma_p}$. The restricted function $\eta_0\vert_{\varphi^{-1}(q)}$ has a local isolated minimum at $p$. For simplicity we assume the local minima to be global after further shrinking of
$\Sigma_p$ and $\Sigma_q$
and
$\rho_0(q)=\eta_0(p)=0$. We choose an open neighborhood
$\Omega\subset \Sigma_p$ which is relatively compact in $\Sigma_p$ and $\partial\Omega$ its boundary in $\Sigma_p$.
The function $\theta^*\rho_0$ is positive on
$\partial\Omega\backslash\theta^{-1}(q)$ while $\eta_0$
is positive on $\partial\Omega\cap\theta^{-1}(q)$.
Due to continuity of $\theta^*\rho_0$ and $\eta_0$
the set $\partial\Omega$ can be covered
by open sets $W_\alpha$ such that on each set
$W_\alpha$ at least one of the functions
$\theta^*\rho_0$ and $\eta_0$ is positive. 
Hence for each $W_\alpha$ there is a
$\delta_\alpha>0$ such that for
each $t\in(0,\delta_\alpha)$
the function
$\hat\rho_t=\rho_0\circ\theta+t\cdot\eta_0$
is positive. As $\partial\Omega$ is compact, we may
pass to a finite subcover $W_{\alpha_i}$ of $\partial\Omega$ and set $\delta=\min_i\delta_{\alpha_i}>0$. Then $\hat\rho_t$
is positive for all $t\in(0,\delta)$ on
$\partial\Omega$.
But as $\rho_t(p)=0$ and 
$\min_{\partial \Omega} \rho_t>0$ the minimum is attained in $\Omega$.
So for each $G$-stable open neighborhood of $G\cdot p$
there is a $\delta>0$ such that
$\rho_t=\varphi^*\rho+t\cdot\eta$ attains a local minimum on some $G$-orbit in this neighborhood.
\end{proof}
\begin{proof}[Proof of Lemma~\ref{lemma_liftingGeneralSliceModelToSpecialLocalSliceModel}]
Recall that first there is a special local slice model
$Z=\gtube\timescat{L}\Delta$
given (compare Lemma~\ref{lemma_isolatedAndUniqueMinimumLocusInClosedCplxfdOrbitInSpecialLocalSliceModel}). There are points $z_0,z_1\in Z$ such that $Z$ decomposes into a closed complexified $G$-orbit ${B(z_0)}=\gtube\timescat{L}\{0\}$
and a complementary open complexified $G$-orbit ${B(z_1)}$.
Second there is a local slice model $Y=\gtube\timescat{K}S$ and a $G$-equivariant holomorphic map
\begin{equation*}
\varphi:\gtube\timescat{L}\Delta
\to \gtube\timescat{K}S\ .
\end{equation*}
Note that $\varphi\vert_{B(z_1)}$ is automatically submersive to its image due to $G$-equivariance.
Thus $\varphi(B(z_1))$ is a
local submanifold of $Y$.
By assumption
there is a $G$-exhaustive strictly plurisubharmonic function $\rho$ on $Y$ attaining a
local minimum at $q\in Y$. We aim to show that the point $q$ is not contained in $\varphi({B(z_1)})$ and therefore assume the contrary $q\in\varphi({B(z_1)})$. Lemma~\ref{lemma_existenceOfExhaustivePotentialOnSliceModel}
provides the existence of a
$G$-exhaustive strictly plurisubharmonic function on
$Z$ which we denote $\eta$. The $G$-stable set
$W=\varphi^{-1}(G\cdot q)\subset {Y}$
is closed, such that
$\eta$ attains a minimum at some point $p$. Without restriction we may choose $p\in\varphi^{-1}(q)$.
Further 
without restriction we may add to $\eta$
a $G$-invariant smooth function $\chi$
having $G\cdot p$ as local isolated
minimum locus
such that the modified $G$-invariant function
stays strictly plurisubharmonic and $G$-exhaustive. For simplicity
we denote this modified function again by $\eta$.
Lemma~\ref{lemma_producingLocalMinimumOnLiftedFunction} shows that for a $G$-stable neighborhood
of $p$ in ${B(z_1)}$
there is a $\delta>0$ such that 
for each $t\in(0,\delta)$ the function
$\rho_t=\varphi^*\rho+t\cdot\eta$ has a $G$-orbit as a local minimum locus in that neighborhood.
Additionally, this function is 
strictly plurisubharmonic as a sum of
a plurisubharmonic and a strictly plurisubharmonic function and as $\eta$ is $G$-exhaustive, so is $\rho_t$. Choose $t_0\in(0,\delta)$ and let $G\cdot p_0$ be this
local minimum locus in ${B(z_1)}$ for the function $\rho_{t_0}$. 
As above we may modify $\rho_{t_0}$
slightly such that the modified function is
strictly plurisubharmonic
and $G$-exhaustive and additionally 
has the orbit $G\cdot p_0$
as a local isolated minimum locus.
Since $p_0\in {B(z_1)}$, this contradicts
Lemma~\ref{lemma_uniqueCriticalOrbitInSpecialLocalSliceModel} and shows
that $q\not\in\varphi({B(z_1)})$.
\end{proof}
\begin{lemma}
Let $Y=\gtube\timescat{K}S$ be a local slice model
and $\rho$ a $G$-exhaustive
strictly plurisubharmonic function with corresponding momentum map $\mu$. Let $\pi:Y\to Y\modmod G$ be the analytic Hilbert quotient.
Then for all $c\in\R$
we have for $U_c=\{y\in X\ \vert \ \rho<c\}$
\begin{enumerate}
\item $\pi(U_c)$ is open
\item $\pi(U_c)=\pi(U_c\cap\mathcal{M})$
where $\mathcal{M}=\mu^{-1}(0)$ and
\item $\begin{array}[t]{lcl}
       \pi^{-1}(\pi(U_c))&=&\{y\in Y\ \vert \ B(y)
       \ \text{intersects}\ U_c\}\\
       &=&\{y\in Y\ \vert \ \overline{B(y)}
       \ \text{intersects}\ U_c\cap\mathcal{M}\}
      \end{array}$
\end{enumerate}
holds.\label{lemma_imageOfSublevelsAreOpenInLocalSliceModel}
\end{lemma}
\begin{proof}
For a given set $U_c$ let
$U_c^B$ be the set of all points whose complexified
$G$-orbit intersets $U_c$. First we will show that
$U_c^B$ is open. In fact, $U_c^B$ is the smallest
set containing $U_c$ such that for all
open subsets $W\subset U_c^B$ such that given $\xi\in\lie{G}$ such that $\varphi_t^\xi(W)\subset Y$
implies $\varphi_t^\xi(W)\subset U_c^B$. Hence $U_c^B$
is open.\\
Now we will prove that $U_c^B$ is $\pi$-saturated.
Proposition~\ref{proposition_fiberStructureInLSM}
states that a $\pi$-fiber
$F$ is the disjoint union of a unique
closed complexified $G$-orbit $E$ and non-closed
complexified $G$-orbits $B_i$ and $E\subset\partial B_i$ for all $i$. As $\rho$ restricted to $F$ is a
$G$-exhaustion, the minimum is attained but not attained on each non-closed complexified $G$-orbit $B_i$. From this we conclude
\begin{equation}
\inf_B\rho=\min_E\rho
\end{equation}
Hence the set $B$ is contained in $U_c^B$ if and only if $E\subset U_c^B$, thus either the fiber $F$ is entirely contained in $U_c^B$ or the sets are disjoint. Hence
$U_c^B$ is $\pi$-saturated. Therefore
$\pi(U_c)$ is open, since $\pi^{-1}(\pi(U_c))=U_c^B$ is open. With the above notation for $B_i$ and $E$, we obtain
$\pi(U_c\cap B)=\pi(U_c\cap E)=\pi(U_c\cap\mathcal{M})$.
\end{proof}
\begin{proof}[Proof of Theorem~\ref{theorem_homeomorphismInLocalSliceModel}]
In the first part of the proof we show that a
$G$-ex\-haus\-tive strictly plurisubharmonic function $\rho$ cannot become critical,
and therefore also
not minimal, on a non-closed complexified $G$-orbit in a
local slice model $Y=\gtube\timescat{K}S$.
Assume the contrary, i.e. let $B\subset Y$ be a non-closed
complexified $G$-orbit in $Y$ with critical point $p\in B$
of $\rho$. Introduce the projection
$\pi_K:\gtube\times S\to Y$ and set $p=\pi_K(\gamma, s)$.
Proposition~\ref{proposition_bijectionOfOrbits} shows that there is a
non-closed complexified $K$-orbit $B_S\subset S$
such that $\pi_K^{-1}(B)=\gtube\times B_S$. 
Recall
$S$ can be realized as an open Runge subset 
of a $K$-representation $V$ constructed as
$S=\{v\in V\ \vert\ \nu_V(v)<c\}$
for some $K$-invariant
strictly plurisubharmonic exhaustion $\nu_V$ on $V$.
Since $B_S$ is
not closed the Hilbert Lemma (cf. for example \cite[III.2.4]{Kra84}) shows that 
there is an element $k\in K$ and 
subgroup $L\subset K$ isomorphic to $S^1$ with an induced group homomorphism
$\lambda^\C:\C^*\cong L^\C\to K^\C$ such
that the following holds.
The $\C^*$-orbit through $s_0=k\cdot s$ closes up in a single fixed point of the $\C^*$-action and this closure
is a 1-dimensional complex subspace of $V$, in particular
$\lim_{t\to -\infty}\lambda^\C(e^t)s_0$ exists.
Define $\Omega=S\cap\overline{\C^*\!\cdot\! s_0}$.
As the map $t\mapsto \nu_V(\lambda^\C(e^t)s_0)$ is strictly convex, $\Omega$ is a connected
1-dimensional (reduced) complex
space, admitting an $L$-action
with a single closed complexified $L$-orbit, namely a fixed point of $L$, and an open complexified $L$-orbit.
We may normalize $\Omega$ to $\Delta$ which is $L$-equivariant and so $\Delta$ is $L$-equivariant biholomorphic to the unit disk in $\C$ with $L$
acting non trivially and linearly on $\C$. The 
$L$-equivariant holomorphic map
$\Delta\to S$ induces a $G$-equivariant holomorphic
map $\varphi:\gtube\timescat{L}\Delta\to
\gtube\timescat{K}S$ with the local minimum point $p$
of $\rho$ in the image,
since $\pi_K(\gamma,s)=\pi_K(\gamma\cdot h^{-1},h\cdot s)$.
Let $z_0, z_1$ be two points in the special local slice model $Z=\gtube\timescat{L}\Delta$.
We have $Z=B(z_0)\dot\cup B(z_1)$
where $B(z_0)=\gtube\timescat{L}\{0\}$ is closed and $B(z_1)$ its open complement.
Restricted to the open
complexified $G$-orbit ${B(z_1)}$ the map $\varphi$ becomes a $G$-equivariant holomorphic 
submersion onto its image in the complexified
$G$-orbit $\gtube\timescat{K}B_S$ 
in $\gtube\timescat{K}S$, since for each $y$ in the image of
$\varphi$ the tangent space $T_y(G\cdot y)$
spans
the complex tangent space to the complexified $G$-orbit.
Lemma~\ref{lemma_liftingGeneralSliceModelToSpecialLocalSliceModel} gives a contradiction to $p$ being 
in the image of $\varphi({B(z_1)})$.
Thus we showed that $\rho$ does not become critical
on a non-closed complexified $G$-orbit.\\
On the other hand, the existence of
two distinct $G$-orbits in a closed $G$-orbit $B\subset Y$
on which $\rho$ becomes critical
is excluded by Corollary~\ref{corollary_uniqueCriticalOrbitInComplexifiedGorbitIfExhaustive}.\\
In summary for every fiber
$F$ of $\pi:Y\to Y\modmod G$
the function
$\rho\vert_F$ 
has to become minimal and critical on the
unique 
closed complexified $G$-orbit $B$. On this complexified $G$-orbit
$\rho\vert_B$ becomes critical
exactly in a single $G$-orbit.
Exactly on this set
$\rho\vert_F$
becomes minimal.
On the non-closed complexified $G$-orbits in $F$ in turn
the function $\rho$ does not become critical at all.
Thus the map
$\varphi:\mathcal{M}/G\to Y\modmod G$ induced from the inclusion $\mathcal{M}\hookrightarrow Y$ is a bijection.\\
We are left to show that $\varphi$ is a homeomorphism.
Observe that $\varphi$ is continuous by construction.
Given a convergent sequence 
$z_n\in Y\modmod G$
Lemma~\ref{lemma_imageOfSublevelsAreOpenInLocalSliceModel}
shows that there is a $c<\sup_Y\rho$ defining
the $G$-stable open set
$U_c=\{y\in Y\ \vert\ \rho(y) <c\}$ and the $G$-stable set $\mathcal{M}_c=U_c\cap \mathcal{M}$ and there
are points $y_n\in \mathcal{M}_c$
such that $\pi(y_n)=z_n$.
As $\rho$ is a $G$-exhaustion, $U_c/G$ is relatively compact in $Y/G$. Hence
$p(\mathcal{M}_c)$
is relatively compact in $\mathcal{M}/G$.
Therefore by possibly passing to a subsequence $p(y_n)$ converges
in $\mathcal{M}/G$. By construction $p(y_n)=\varphi^{-1}(z_n)$, hence $\varphi$ is open and therefore a
homeomorphism.
\end{proof}\label{subsection_localHomeomorphism}
\subsection{Local biholomorphism to local analytic Hilbert quotient}
As mentioned in the introduction we endow
the symplectic
reduction $\mathcal{M}/G$ 
with a structure sheaf $\osheaf_{\mathcal{M}/G}$
in the following way. Denote by
$p:\mathcal{M}\to\mathcal{M}/G$ 
the quotient map. Then for an open subset
$Q\subset\mathcal{M}/G$
the space of sections
$\mathcal{O}_{\mathcal{M}/G}(Q)$
consists of those maps $f:Q\to\C$ such that
$p^*f:p^{-1}(Q)\to\C$ extends to a holomorphic
function on some neighborhood of $p^{-1}(Q)\subset Y$.
So $({\mathcal{M}/G},\osheaf_{\mathcal{M}/G})$
is a ringed space.
\begin{lemma}
Let $Y=\gtube\timescat{K}S$ be a local slice model, $\rho$ a $G$-exhaustive
strictly plurisubharmonic function
inducing a K\"ahler form $\omega$, a momentum map $\mu$
and the zero level $\mathcal{M}=\mu^{-1}(0)$. 
Let $V\subset Y$ be an open subset
such that $\mathcal{M}_V=\mathcal{M}\cap V$
is $G$-stable and let $f:V\to\C$ be
a holomorphic function such that
$f\vert_{\mathcal{M}_V}$ is $G$-in\-vari\-ant. Then
there is a unique
holomorphic function $h$ on $U=\pi^{-1}(\pi(\mathcal{M}_V))$ which is constant on the $\pi$-fibers such that
$f\vert_{\mathcal{M}_V}=h\vert_{\mathcal{M}_V}$
holds.\label{lemma_holomorphicFunctionExtendsToSaturatedSet}
\end{lemma}
\begin{proof}
For $x\in\mathcal{M}_V$ choose a local slice model $Y_0=\gtube\timescat{K_0}S_0\subset V$
around $x$.
Let $A$ be the smallest 
closed analytic subset
in $Y_0$ containing $\mathcal{M}_0=\mathcal{M}\cap~Y_0\subset V$.
This set $A$ is $G$-stable and in particular contains all closed
complexified $G$-orbits in $Y_0$ which intersect
$\mathcal{M}_0$.
The restriction
$f\vert_{A}$ is $G$-invariant and 
therefore the restriction $f\vert_{A_S}$ to the 
$K_0$-stable
intersection $A_S=A\cap S_0$ is $K_0$-invariant.
Since $S_0$ is a Stein manifold, $f\vert_{A_S}$
extends to $S_0$ holomorphically.
This extension can be made $K_0$-invariant by averaging. We may consider this function
as a holomorphic function on the quotient
$S_0\modmod K_0$. Via the identification
$Y_0\modmod G\cong S_0\modmod K_0$
the function can be pulled back to a $G$-invariant holomorphic function $h_0$ on $Y_0$.
On $A$ and hence on $\mathcal{M}_0$ the functions
$f$ and $h_0$ coincide.\\
This process can be made for every point
in $\mathcal{M}_V$. Since the extensions are unique,
these extensions glue together
to a $G$-invariant holomorphic function
on a $G$-stable neighborhood $V_0\subset U$
of $\mathcal{M}_V$, in particular this function is constant on the $\pi\vert_{V_0}$-fibers.
This function extends uniquely to
a function $h:U\to\C$ which is constant
on the $\pi$-fibers and coincides with $f$ on
$\mathcal{M}_V$.
We claim that this function is holomorphic. 
In order to see this let $W\subset V$ consists of those points $w$
which are mapped
by $\pi$ to
a non-singular point such that the derivative
$\pi_*$ at $w$ has full rank. This set 
$W$ is the union of complexified $G$-orbits
and $V\backslash W$ is a proper analytic
subset. Restricted to $W$ the map $\pi$ is
submersive to its image. 
The set of points at which $h$ is not holomorphic
is a union of connected components of 
$\pi$-fibers. But
for any point
$w\in W$ each connected component of the complexified orbit $B(w)$ intersects
$V_0$. Thus $h\vert_W$ is holomorphic
and continuous on $V$ hence holomorphic on $V$.
\end{proof}
\begin{lemma}
Let $Y=\gtube\timescat{K}S$ be a local slice model, $\rho$ a $G$-exhaustive
strictly plurisubharmonic function
inducing a K\"ahler form $\omega$, a momentum map $\mu$
and the zero level $\mathcal{M}=\mu^{-1}(0)$. Then
the homeomorphism $\varphi:\mathcal{M}/G\to Y\modmod G$
established in Theorem~\ref{theorem_homeomorphismInLocalSliceModel} induces an isomorphism of the sheaves 
$\osheaf_{\mathcal{M}/G}$ and $\osheaf_{Y\modmod G}$.\label{lemma_sheafIsomInLocalSliceModel}
\end{lemma}
\begin{proof}
Restriction induces a sheaf morphism
$\mathcal{O}_{Y\modmod G}\to \varphi_*\mathcal{O}_{\mathcal{M}/G}$
in a trivial way.
In order to show that this is an isomorphism
we will define an inverse. Recall the notation of the quotient maps $p:\mathcal{M}\to\mathcal{M}/G$
and $\pi:Y\to Y\modmod G$.
For an open set $R\subset Y\modmod G$ we set $Q=\varphi^{-1}(R)$. An element
in $\varphi_*\mathcal{O}_{\mathcal{M}/G}(R)$
is given by an open subset $V\subset Y$ such that
$\mathcal{M}_V=\mathcal{M}\cap V$ is $G$-stable
and $p(\mathcal{M}_V)=Q$ together with a holomorphic function $f:V\to\C$ which is $G$-invariant on
$\mathcal{M}_V$. Recall that 
$R=\pi(\mathcal{M}_V)$
and denote $U=\pi^{-1}(R)$.
On this saturated open set
there is
a unique holomorphic function $h:U\to\C$
constant on the $\pi$-fibers
which coincides with $f$ on $\mathcal{M}_V$
(Lemma~\ref{lemma_holomorphicFunctionExtendsToSaturatedSet}).
This holomorphic function can be seen as an element
of $\mathcal{O}_{Y\modmod G}(R)$.
By this extension process the
sheaf morphism 
$\varphi_*\mathcal{O}_{\mathcal{M}/G}\to \mathcal{O}_{Y\modmod G}$
is defined which is inverse to the sheaf morphism
arising from restriction.
Thus with these structure sheaves the homeomorphism $\varphi$
becomes a biholomorphism.
\end{proof}
Now we will summarize the partial results to prove the
central Theorem~\ref{theorem_quotientLocallyGivenFromHilbertQuotient}.
It states that each point $x\in\mathcal{M}$
admits a $G$-stable neighborhood $U$ such that
with $\mathcal{M}_U=\mathcal{M}\cap U$
there is an induced homeomorphism
$\varphi:\mathcal{M}_U/G\to U\modmod G$
which induces an isomorphism
between the sheaves
$\mathcal{O}_{\mathcal{M}_U/G}$
and
$\osheaf_{U\modmod G}$.
\begin{proof}[Proof of Theorem~\ref{theorem_quotientLocallyGivenFromHilbertQuotient}]
For a given a point $x\in\mathcal{M}$ Theorem~\ref{theorem_slice} provides
a neighborhood $Y$ isomorphic to a local slice model with quotient map
$\pi:Y\to Y\modmod G$.
Theorem~\ref{theorem_localPotentialAroundM}
implies
the existence of a local potential $\rho$
on $Y$. 
For a given point there is
a $\pi$-saturated open neighborhood $U\subset Y$ 
of $x$
such
that $\pi(U)\subset Y\modmod G$ is relatively compact and Runge.
Then
there is a $G$-stable open 
neighborhood $V$ of $\mathcal{M}\cap U$
such that $V/G$ is relatively compact in $Y/G$.
Lemma~\ref{lemma_bendingUpPotential}
shows that there is
a strictly plurisubharmonic $G$-exhaustion
$\tilde\rho$ with $\tilde\rho\vert_V=\rho\vert_V$
inducing a momentum zero level
$\widetilde{\mathcal{M}}$ with 
$\widetilde{\mathcal{M}}\cap U=\mathcal{M}\cap U$.
The induced map 
$\tilde\varphi:\widetilde{\mathcal{M}}/G\to Y\modmod G$
is shown to be a homeomorphism (Theorem~\ref{theorem_homeomorphismInLocalSliceModel})
inducing
an isomorphism of sheaves (Lemma~\ref{lemma_sheafIsomInLocalSliceModel}).
Its restriction
\begin{equation}
\varphi:\mathcal{M}_U/G\to \varphi(\mathcal{M}_U/G)
=\pi(U)
\subset Y\modmod G 
\end{equation}
is a homeomorphism inducing an isomorphism of sheaves
making $\mathcal{M}_U/G$ a reduced normal complex space
in a natural way.
According to Lemma~\ref{lemma_holomorphicFunctionExtendsToSaturatedSet} a $G$-invariant holomorphic function on $U$
is uniquely given by its restriction
to $\mathcal{M}_U$ and in particular constant on the $\pi$-fibers. Therefore $\pi(U)=U\modmod G$, hence
\begin{equation*}
\varphi:\mathcal{M}_U/G\to U\modmod G
\end{equation*}
is a biholomorphism.
\end{proof}
Theorem~\ref{theorem_quotientComplexStructure} stating that 
$(\mathcal{M}/G,\mathcal{O}_{\mathcal{M}/G})$ is a reduced normal complex space
is now a consequence of Theorem~\ref{theorem_quotientLocallyGivenFromHilbertQuotient}.
\begin{proof}[Proof of Theorem~\ref{theorem_quotientComplexStructure}]
Theorem~\ref{theorem_quotientLocallyGivenFromHilbertQuotient} shows that
$(\mathcal{M}/G,\mathcal{O}_{\mathcal{M}/G})$ is locally biholomorphic
to a reduced normal complex space. Since
$\mathcal{M}/G$ is a Hausdorff
space due to the properness of the action, the entire space
arises from gluing reduced normal complex spaces.
\end{proof}\label{subsection_localBiholomorphism}

%%%%%%%%%%%%%%%%%%%%%%%%%%%%%%%%%%%%%%%%%%%%%
%
%   6   Kähler structure on symplectic reduction
%
%%%%%%%%%%%%%%%%%%%%%%%%%%%%%%%%%%%%%%%%%%%%%

\section{K\"ahler structure on symplectic reduction}\label{section_KaehlerStructure}
In this section we will show that the
symplectic reduction $\mathcal{M}/G$
inherits a K\"ahler structure.\\
The space $X$ and hence the
momentum zero level $\mathcal{M}$ can be given a stratification into smooth parts if we compound
all points of $\mathcal{M}$ having the same isotropy group $H\subset G$ up to conjugation in a
stratum $S_H\subset \mathcal{M}$. The sets
$S_H/G$ induce a stratification of the symplectic reduction $\mathcal{M}/G$ into smooth parts.
The restrictions $p\vert_{S_H}:S_H\to S_H/G$ are fiber bundles.
We call this the
stratification by $G$-orbit type.\\
The situation of the K\"ahler form on each stratum $S_H/G$ is rather simple. To expain this,
choose a stratum $S=S_H$ and
for some point $x\in S\subset\mathcal{M}$ restrict the bilinear form $\omega_x$
to $T_xS$. The vector space $T_xS$ splits into $T_x(G\cdot x)$ and $T_xS\cap JT_xS$,
where $J$ denotes the complex structure.
The form is 
non-degenerate on the latter (complex) linear space while $T_x(G\cdot x)$ is the kernel of the restricted form.
Thus given any local section $\sigma$ to $p_S:S\to S/G$ the pullback $\sigma^*\omega$
is closed and non-degenerate, hence symplectic. Furthermore $\sigma^*\omega=\de\dc\sigma^*\rho$, since
$\sigma^*\dc\rho=\dc\sigma^*\rho$, thus $\sigma^*\omega$ is a K\"ahler form on $S/G$.
In order to define a K\"ahler structure
on the usually singular complex space $\mathcal{M}/G$ across the strata we have the following definition.
\begin{definition}\label{definition_KaehlerStructure}
 A {\em K\"ahler structure} on a complex space is given by an open covering $\{U_\alpha\}$ and strictly plurisubharmonic functions
 $\rho_\alpha$ on $U_\alpha$ such that \mbox{$\rho_\alpha-\rho_\beta$} is pluriharmonic on $U_\alpha\cap U_\beta$, i.e.
 locally the real part of some holomorphic function.
 A plurisubharmonic function $\rho$ is said 
 to be {\em strictly} plurisubharmonic if the plurisubharmonicity is stable under
 perturbation, i.e. if for each smooth function $h$ and each relatively compact open set $U$
 there is a
 $\delta>0$ such that $\rho+t\cdot h$ is plurisubharmonic for all $t$ with
 $\vert t\vert <\delta$ on $U$.
\end{definition}
\begin{remark}
The different notions of plurisubharmonicity across singularities are known to coincide (\cite{ForNar80}).
\end{remark}
It cannot be hoped that the
induced functions $\rho_\alpha$ on $\mathcal{M}/G$ are smooth even if $\omega$ is
smooth on $X$ and $G$ is connected. Consider the simple example (\cite{HHL94})
of the $S^1$-representation
on $\C^2$, $(t,z_1,z_2)\mapsto (tz_1,t^{-1}z_2)$
with (analytic) Hilbert quotient
$\pi:\C^2\to\C^2\modmod S^1$. 
The quotient can be realized by the
biholomorphism
$\psi:\C^2\modmod S^1\to \C, [z_1,z_2]\mapsto z_1z_2$.
The function
$\rho(z_1,z_2)=\vert z_1\vert^2+\vert z_2\vert^2$
defines an $S^1$-invariant strictly plurisubharmonic function 
on $\C^2$ inducing $\omega$, $\mu$
and $\mathcal{M}=\{(z_1,z_2)\in\C^2\ \vert\ \vert z_1\vert=\vert z_2\vert\}$. Hence there is a function
$\theta:\C\to\R$ such that
$\theta\circ\psi\circ\pi\vert_\mathcal{M}
=\rho\vert_\mathcal{M}$.
But $\theta(w)=2\vert w\vert$ is not differentiable.\\
Corollary~\ref{corollary_KaehlerStructure} states the existence of a K\"ahler form in the sense of the
above Definition~\ref{definition_KaehlerStructure}. 
\begin{proof}[Proof of Corollary~\ref{corollary_KaehlerStructure}]
For every point in the momentum zero level
$\mathcal{M}$ 
there is a local slice model with a $G$-invariant potential (Theorem~\ref{theorem_slice}). 
Let $\mathcal{M}=\bigcup Y_\alpha$
be a covering 
of $\mathcal{M}$
by local slice models $Y_\alpha$ with $G$-invariant potentials $\rho_\alpha$. 
This provides a cover of $\mathcal{M}/G$ by open sets $Z_\alpha=p(Y_\alpha\cap\mathcal{M})$.
Restricted to $Y_\alpha\cap\mathcal{M}$ the function $\rho_\alpha$ is continuous and $G$-invariant, thus
pushes down to continuous functions $\rho^Z_\alpha:Z_\alpha\to\R$, i.e. 
$\rho^Z_\alpha\circ p=\rho_\alpha$ on $Y_\alpha\cap\mathcal{M}$. We claim that these functions
$\rho^Z_\alpha$ form the K\"ahler structure.\\
First we show that these functions glue together in the expected way.
For a point
$x\in\mathcal{M}\cap Y_\alpha\cap Y_\beta$
choose a $G$-stable neighborhood
$\Omega\subset Y_\alpha\cap Y_\beta$
such that
$\pi(\Omega)=\pi(\mathcal{M}\cap\Omega)$
and every complexified $G$-orbit closes up
in $\mathcal{M}$.
This is possible just by choosing
a local slice model $\Omega_0\subset Y_\alpha\cap Y_\beta$ around $x$ and setting
$\Omega=\pi^{-1}(\pi(\mathcal{M}\cap\Omega_0))$.
The difference $\delta=\rho_\alpha-\rho_\beta$
is pluriharmonic and $G$-invariant thus
$J\tilde v(\delta)$ is constant on
each $\pi$-fiber for all $v\in\lie{G}$.
By construction
$J\tilde v(\delta)\vert_{\mathcal{M}\cap \Omega}=0$ and each fiber intersects
$\mathcal{M}\cap \Omega$, thus the derivatives
$J\tilde v(\delta)$ vanishes and hence
$\delta$ is constant on each $\pi$-fiber.
Locally $\delta$ is the real part of some holomorphic function $f$ which has to be constant
on the $\pi$-fibers as well. This function is the 
pullback of a holomorphic function $g$ on an open
subset around $\pi(x)$. Here $\rho^Z_\alpha-\rho^Z_\beta$ is the real part of $g$.
Thus the functions $\rho^Z_\alpha$ glue together in the desired way.\\
It remains to show that 
$\rho^Z_\alpha$ defines a strictly
plurisubharmonic function on $Z_\alpha$.
In the sequel we will drop the index.
The stratification by $G$-orbit type induces a stratification
on $\mathcal{M}/G$.
As explained above in this section
on each stratum the functions $\rho^Z$
are known to be strictly plurisubharmonic.
We claim that their extensions 
across the strata are strictly plurisubharmonic.
We may check plurisubharmonicity
by pulling back the functions via a
holomorphic map from the unit disk (\cite{ForNar80}).
There, they extend to subharmonic functions.
This result is shown in a much more general setting in \cite{GraRem56}.\\
Finally in order to see
that $\rho$ is strictly plurisubharmonic choose for some point $y\in \mathcal{M}\cap Y$ a relatively compact
open subset $U_Y\subset Y$. Its image $U_Z=\pi(U_Y)$ is open and relatively compact.
Let $h$ be a
smooth function on $Z$.
There is some $\delta>0$ such
that $\rho_t=\rho+t\cdot\pi^*h$ is strictly
plurisubharmonic on $U_Y$ for all $\vert t\vert <\delta$ since
$\rho$ is strictly
plurisubharmonic. Moreover, as these functions $\rho_t$ are $G$-invariant
and strictly plurisubharmonic, they define
momentum maps $\mu_t$ on $Y$
and $\mu_t^{-1}(0)=\mathcal{M}\cap Y$ holds.
Therefore the pushed down functions equal $\rho^Z+t\cdot h$. By the same arguments as above,
they are plurisubharmonic. Hence $\rho^Z$ is
strictly plurisubharmonic on $Z$ which finishes the proof.
\end{proof}

%%%%%%%%%%%%%%%%%%%%%%%%%%%%%%%%%%%%%%%%%%%%%
%
%   7   Examples without invariant potential
%
%%%%%%%%%%%%%%%%%%%%%%%%%%%%%%%%%%%%%%%%%%%%%

\section{Examples without invariant potential}\label{sect_noPotential}
Finally we give examples where there is no $G$-invariant potential. We found cases
for a solvable group and semisimple groups.
\subsection{Examples with large isotropy group}
Recall that given a (real) manifold $M$ with 
a proper action of $G$, a closed $G$-invariant 2-form $\tau$ with momentum map $\nu$,
this setting can be complexified to a
``Stein extension'' in the following way.
There is a Stein $G$-manifold $X$, a $G$-invariant K\"ahler form $\omega$ and a momentum map $\mu$
on $X$ and a totally real embedding $i:M\to X$ with
$\dim_\R M=\dim_\C X$ such that
$\tau=i^*\omega$ and
$\nu=i^*\mu$ (\cite{Str01}).
The idea is that given a $G$-invariant potential $\rho$ the restriction of the 1-form $\dc\rho$ to $M$ is a $G$-invariant 1-form $\beta$ with
$\de\beta=\tau$. But if the isotropy group is ``too large'', there is no non-vanishing invariant 1-form.
\begin{proposition}
Let $K$ be a compact subgroup of $G$, $\tau$ a
non-vanishing $G$-invariant 2-form with momentum map $\nu$ on $M=G/K$.
The group $K$ acts on the cotangent space $T^*_{[e]}M$ where $e$ is the neutral element of $G$ and $[e]\in G/K$ its image
under the projection. If the only $K$-fixed point in $T^*_{[e]}M$ is $0$, then
$M$ admits no $G$-invariant 1-form $\beta$ with $\de\beta=\tau$
and therefore the complexification of $M$ admits no $G$-invariant potential.\label{prop_largeIsotropy}
\end{proposition}
\begin{proof}
 Given a $G$-invariant 1-form $\beta$ then $\beta_e=0$ in $T^*_{[e]}M$ and
 by $G$-invariance $\beta=0$. Since $\tau$ is not vanishing entirely,
 $\de\beta\not=\tau$.
\end{proof}
From this observation we obtain explicit examples.
\subsection{Case of semisimple groups}
%What is correct typesetting of SU(n)?
Let $G=SU(n,1)$ act by holomorphic transformations transitively on the ball $B=\{x\in\C^n\ \vert\ \lVert x\rVert< 1\}$ with isotropy $K=SU(n)$.
The K\"ahler form $\tau$ corresponding to the Bergman metric is kept invariant by $G$ and is exact. By construction
of the metric one can see that $\tau$ admits a momentum map $\nu$. The only fixed point of the $K$-action on
$T_0B$ is $0$. Now consider $B$
as a real manifold and let $X$ be a Stein extension.
All conditions of Proposition~\ref{prop_largeIsotropy} are satisfied, so there is no $G$-invariant potential on the complexified
spaces.
\subsection{Case of solvable groups}
Let $G$ be the group of isometries of $\R^2$
with euclidean metric, $G=S^1 \ltimes \R^2$, and 
extend the action to $\C^2$ such that $G$ acts by holomorphic transformations.
Within the standard coordinates on $\C^2$, $z_i=x_i+iy_i$,
the K\"ahler form
$$\omega=\de x_1\wedge\de y_1 + \de x_2\wedge\de y_2$$
is $G$-invariant. We fix a basis for the Lie algebra $\lie{G}$
$$\derv{x_1},\derv{x_2},\derv{\varphi}=x_2\derv{x_1}-x_1\derv{x_2}+y_2\derv{y_1}-y_1\derv{y_2}\ .$$
The Lie algebra structure with respect to this basis is given by
\begin{align*}
\left[\derv{\varphi},\derv{x_1}\right]&=\derv{x_2}\\
\left[\derv{\varphi},\derv{x_2}\right]&=-\derv{x_1}\\
\left[\derv{x_1},\derv{x_2}\right]&=0
\end{align*}
We have
\begin{align*}
 \imath_\derv{x_i}\omega&=\de y_i\\
 \imath_\derv{\varphi}\omega&=x_2\de y_1-y_2\de x_1-x_1\de y_2+y_1\de x_2\\
& =\de(-x_1y_2+x_2y_1)
\end{align*}
If $\mu$ is a momentum map, then
\begin{align*}
 \mu^\derv{x_i}&=y_i+c_i\\
 \mu^\derv{\varphi}&=x_2y_1-x_1y_2+c
 \end{align*}
for some constants $c_1,c_2,c\in\R$.
Equivariance of $\mu$ gives
$\imath_{\tilde v}\de\mu^w=\mu^{[v,w]}$ 
for all $v,w\in\lie{G}$.
This implies
$$\imath_\derv{\varphi}\de\mu^\derv{x_1}=y_2=y_2+c_2$$
and therefore $c_2=0$ and
analogously $c_1=0$.
This shows that
for each $c\in\R$ there is a momentum
map $\mu_c$.\\
The function
$$\rho_0=\frac{1}{2}(y_1^2+y_2^2)$$
is a $G$-invariant potential for $\mu_0$.
Suppose there is a potential $\rho_c$ for $\mu_c$.
Then $\delta=\rho_c-\rho_0$ is a $G$-invariant pluriharmonic function. Since $\delta$ is $\R^2$-invariant
and pluriharmonic, there are constants $a, a_1,a_2$ such that
$\delta=a_1y_1+a_2y_2+a$. The $S^1$-invariance applied to $\delta$ shows that
$\de\delta(0)=0$, thus $\delta$ is constant.\\
This shows that for $\mu_c$, $c\not=0$,
there is no invariant potential.

\bibliographystyle{alpha}

{\footnotesize}
%\textsc{Department of Mathematics, Ruhr-Universit\"at Bochum, 44780 Bochum, Germany}
\textsc{Department of Mathematics\\
Ruhr-Universit\"at Bochum\\
44780 Bochum\\
Germany}

\medskip
P.~Heinzner\\
\textit{E-mail address} \texttt{peter.heinzner@ruhr-universitaet-bochum.de}

\smallskip
B.~Stratmann\\
\textit{E-mail address} \texttt{bernd.x.stratmann@ruhr-universitaet-bochum.de}
\end{document}